\documentclass[11pt]{amsart}
\usepackage{amsmath,amsfonts,comment,amsthm,euscript,dirtytalk,mathtools,amssymb,graphicx,mathrsfs,enumitem, cleveref, tikz, pgfplots}
\pgfplotsset{compat=1.8}
\usepackage{accents}

\usepackage{bm}
\usepackage{mathtools}
\usepackage{amsmath,enumitem}
\usepackage{amsthm}
\usepackage{amsfonts}
\usepackage{mathtools}
\usepackage{amssymb}
\usepackage{graphicx}
\graphicspath{ {./} }

\usepackage{biblatex}
\address{Department of Mathematics, Johns Hopkins University, Baltimore, MD 21218}
\email{zzhan296@jh.edu}

\newcommand{\la}{\lambda}

\newtheorem*{notation}{Notation}

\newtheorem{theorem}{Theorem}

\newtheorem{lemma}{Lemma}[section]
\newtheorem*{remark}{Remark}
\newtheorem{proposition}{Proposition}[section]

\begin{document}
\title{Strichartz and Spectral Projection Estimates on Asymptotically Conic Manifolds}
\author{Zhexing Zhang}
\date{}

\begin{abstract}
 We prove the lossless unit interval Strichartz theorem on asymptotically conic surfaces, assuming that a large enough neighborhood of its trapped set has negative curvature. 
 
 We also prove the spectral projection theorem for $6<q<\infty$ on asymptotically Euclidean surfaces with nonpositive curvature, assuming a large enough neighborhood of its trapped set has negative curvature. If the negatively curved region extends to a sufficiently far-out cutoff, we also obtain the critical and subcritical estimates, without a curvature sign assumption near infinity.
\end{abstract}
\maketitle
\section{Introduction}
Assume $(M,g)$ is an asymptotically conic manifold of the form
$M=M_0\cup M_\infty$, where $M_0$ is compact and
$M_\infty=[A,\infty)\times\partial M$. On $M_\infty$, assume that
\begin{align}\label{eq:acmetric}
g
=
dr^2+r^2h(r^{-1},\theta,d\theta)
=
dr^2+\sum_{i,j}g_{ij}(r,\theta)d\theta_i d\theta_j,
\end{align}
where
\[
h(x,\theta,d\theta)
=
\sum_{i,j}h_{ij}(x,\theta)d\theta_i d\theta_j
\]
is a smooth family of metrics on $\partial M$ up to $x=0$, and
\[
g_{ij}(r,\theta)=r^2h_{ij}(r^{-1},\theta).
\]

Let $\Delta_g$ denote the nonpositive Laplacian operator on $M$, and let
$P=\sqrt{-\Delta_g}$. Define
\[
p_g:=\det(g_{ij})^{1/2}.
\]
We define an effective potential
\begin{align*}
q(r,\theta)
:=
(2p_g)^{-2}(\partial_r p_g)^2
+(2p_g)^{-2}\sum_{i,j}
\frac{\partial p_g}{\partial\theta_i}
\frac{\partial p_g}{\partial\theta_j}g^{ij}
+2^{-1}p_g\Delta_{M_{\infty}}(p_g^{-1}).
\end{align*}
Here, we let
\begin{align}
\Delta_{M_{\infty}}
:=
(-\Delta_g)|_{M_{\infty}}
=
-\partial_r^2-\frac{\partial_rp_g}{p_g}\partial_r
+\Delta_{\partial M_\infty},
\end{align}
where
\begin{align*}
\Delta_{\partial M_\infty}
=
-p_g^{-1}\sum_{i,j}
\partial_{\theta_i}(p_g g^{ij}\partial_{\theta_j}).
\end{align*}
We now make the same assumptions as (1.2) and (1.3) of \cite{Cardoso2002UniformEO}.
Assume there exists a constant $C>0,$ such that \begin{align}\label{assumption1}
    |q(r,\theta)|<C\,\,\,,\frac{\partial q}{\partial r}<Cr^{-1-\delta_0}
\end{align}
for some $\delta_0>0$ and \begin{align}\label{assumption2}
    -\frac{\partial}{\partial r}\sum_{i,j}g^{ij}(r,\theta)\xi_i\xi_j\geq \frac{C}{r}\sum_{i,j}g^{ij}(r,\theta)\xi_i\xi_j
\end{align}
for $(\theta,\xi)\in T^*\partial M.$ Notice that $\eqref{assumption2}$ implies that on $M_\infty,$ the geodesic flow is convex in terms of $r$.

Meanwhile, notice that in 2-dimensional surfaces, assume the metric is $$dr^2+f(r,\theta)^2d\theta^2$$ on $M_\infty$ for some $f\in C^\infty([A,\infty)\times\partial M)$, then \eqref{assumption2} is equivalent to \begin{align}\label{assumption3}
    -\frac{\partial f(r,\theta)^{-2}}{\partial r}\geq \frac{C}{r}f(r,\theta)^{-2}.
\end{align}

Burq, Guillarmou and Hassell \cite{BGH} proved the Strichartz estimates on a large class of asymptotically conic manifolds with hyperbolic trapped set and satisfying the pressure condition. Recently, Huang, Sogge, Tao and the author \cite{HSTZ} proved the Strichartz estimates on the unit time interval on asymptotically hyperbolic surfaces with negative curvature and bounded geometry without assuming the pressure condition. 

By constructing an asymptotically hyperbolic background manifold of $M$, which agrees with $M$ on $M_0$, we can use the result in \cite{HSTZ} to prove the following Strichartz estimate without assuming the pressure condition. We denote the Laplace operator on $M$ associated with the metric $g$ by $\Delta_g$. 
\begin{theorem}\label{Strichartzthm}
    Assume $(M,g)$ is a 2-dimensional asymptotically conic surface in the form $M=M_0\cup M_\infty$ as above, satisfying \eqref{assumption1} and \eqref{assumption3}. Meanwhile, assume $M_0$ has negative sectional curvature. Then, we have the following Strichartz estimate.
    \begin{align}\label{eq:thmstrichart}
    \|e^{-it\Delta_g}u_0\|_{L^p_tL^q_x(M\times [0,1])}\leq C \|u_0\|_{L^2(M)},
    \end{align}
    for $(p,q)$ satisfying the Keel-Tao condition,
    \begin{align}\label{keeltao}
        2(1/2-1/q)=2/p, \, \, 
\,\,\,\,\,\text{ for } p\in (2,\infty).
    \end{align}
\end{theorem}

Stein \cite{Stein} and Tomas \cite{Thomas} proved the spectral measure estimates for Euclidean spaces. Sogge \cite{sogge88} proved the sharp spectral projection with spectral window of unit size on compact manifolds. Huang and Sogge \cite{SoggeHuangQuasimode2025} proved the spectral projection in logarithmic scale for compact manifolds with nonpositive curvature. Then, Huang, Sogge, Tao and the author \cite{HSTZ} proved the spectral projection in logarithmic scale for manifolds with nonpositive curvature and bounded geometry, as well as the spectral projection estimates on even asymptotically hyperbolic surfaces with strictly negative curvature. 

Let $\mathbf{1}_S$ be the indicator function on the set $S.$ Let $I\subset\mathbb{R}$ be an interval. Then, $\mathbf{1}_{I}(\sqrt{-\Delta_g})$ is the spectral projection operator on the spectral window $I$. Let $S^*M$ be the cosphere bundle of $M$ and denote the principal symbol of $P$ by $p(x,\xi)$. Let $(x(t),\xi(t))=e^{tH_p}(x,\xi)$, where $e^{tH_p}$ denotes the geodesic flow on the cotangent bundle. Define \begin{align}
     \Gamma_\pm:=\{(x,\xi)\in S^*M:x(t)\not\to\infty \text{ as }t\to\mp\infty\}.
 \end{align}
We denote $$K=\Gamma_+\cap\Gamma_-.$$ Define $\pi:S^*M\to M$ with $\pi(x,\xi)=x.$ The trapped set of $M$ is defined to be $\pi(K)$.

Guillarmou, Hassell and Sikora
\cite{Guillarmou2010RestrictionAS} proved restriction estimates on
asymptotically conic manifolds in dimension $n\geq3$. Chen
\cite{xichen} later completed the abstract microlocal Stein--Tomas
argument used in the presence of a microlocal partition of unity.

For the spectral projection results below, we make the additional
assumption that $M$ has asymptotically Euclidean ends in the sense of
\cite[Section 8]{Guillarmou2010RestrictionAS}. More precisely, each
end is diffeomorphic to
\[
(r_0,\infty)\times\mathbb S^1,
\]
and the metric on each end is of the form
\begin{align}\label{eq:aemetric}
g=dr^2+r^2h(\theta,d\theta,r^{-1}),
\end{align}
where $h(\theta,d\theta,x)$ is smooth up to $x=0$.

We explain in Subsection 3.3 why the high-energy restriction argument
needed here also applies in dimension two. We may then follow
\cite{HSTZ} and \cite{Guillarmou2010RestrictionAS} to obtain the
following spectral projection estimate.
\begin{theorem}\label{specthm}
    Assume $(M,g)$ is a 2-dimensional asymptotically Euclidean surface with nonpositive curvature. Meanwhile, $M=M_0\cup M_\infty$ as above, satisfying \eqref{assumption1} and \eqref{assumption3}, and $M_0$ has negative sectional curvature. For $2<q<\infty$, denote
\begin{equation}\label{ii.2}
\mu(q)=
\begin{cases}
2(\tfrac12-\tfrac1q)-\tfrac12, & q\geq6,\\
\tfrac12(\tfrac12-\tfrac1q), & 2<q<6.
\end{cases}
\end{equation}
Then, we have the following spectral projection estimate.
    \begin{align}\label{eq:stri-nontrap}
    \|\mathbf{1}_{[\la,\la+\delta]}(\sqrt{-\Delta_g})f\|_{L^q(M)}
    \lesssim \la^{\mu(q)}\delta^{1/2}\|f\|_{L^2(M)},
    \qquad 6<q<\infty,
    \end{align}
    if $\la\gg1$ and $\delta\in(0,1]$.
\end{theorem}

Notice that the powers of $\la$ and $\delta$ in Theorem~\ref{specthm} agree with the Stein--Tomas restriction estimate in the supercritical range, \cite{Stein} and \cite{Thomas}.

If the negative curvature region covers the support of $1-\chi_\infty$, we can take $\chi_0=1-\chi_\infty$ and there is no $\chi_c$ term. In this case, the endpoint and subcritical estimates also hold, and we do not need nonpositive curvature on the whole surface.
\begin{theorem}\label{specthm-full}
Assume $(M,g)$ is a 2-dimensional asymptotically Euclidean surface of the form $M=M_0\cup M_\infty$ as above, satisfying \eqref{assumption1} and \eqref{assumption3}. Choose a real cutoff $\chi_\infty$ as in Proposition~\ref{new}, depending only on $r$ and nondecreasing on each end, with $0\leq\chi_\infty\leq1$, $\chi_\infty=0$ near $M_0$, and $\chi_\infty=1$ near infinity. Assume $M$ has negative sectional curvature on a neighborhood of $\operatorname{supp}(1-\chi_\infty)$. Then, with $\mu(q)$ as in \eqref{ii.2}, we have
\begin{equation}\label{eq:spec-full-range}
\|\mathbf{1}_{[\la,\la+\delta]}(P)f\|_{L^q(M)}
\lesssim
\begin{cases}
\la^{\mu(q)}\delta^{1/2}\|f\|_{L^2(M)}, & 6\leq q<\infty,\\
(\la\delta)^{\mu(q)}\|f\|_{L^2(M)}, & 2<q<6,
\end{cases}
\end{equation}
for $\la\gg1$ and $0<\delta\leq1$. No curvature sign is assumed outside this neighborhood.
\end{theorem}
The cutoff in Theorem~\ref{specthm-full} must be chosen sufficiently far out for Proposition~\ref{new}. The negative curvature assumption includes the transition region where $0<\chi_\infty<1$.

In Section 2, we prove the Strichartz theorem on asymptotically conic surfaces. In Section 3, we prove the spectral projection theorems, Theorem~\ref{specthm} and Theorem~\ref{specthm-full}.

\begin{notation}
    For any nonnegative quantity $A$ and $B$ depending on $\lambda\geq 1$ and $\delta\leq 1,$ $A\lesssim B$ and $A=O(B)$ both mean $A\leq cB$ for some constant $c>0$ independent of $\lambda$ and $\delta$. We use $A \sim B$ to denote $A\lesssim B$ and $B\lesssim A.$
\end{notation}

\section*{Acknowledgment} 
The author sincerely thanks Xiaoqi Huang, Connor Quinn and Christopher Sogge for their helpful discussion and comments.

\section{Strichartz estimates}
\subsection{Construction of background manifold}

Since curvature is continuous, after a fixed rescaling of the metric, we may assume $M$ has negative sectional curvature on $\{(r,\theta)\in [A,A+1]\times\partial M\}\subset M_\infty$ as well. We construct an even asymptotically hyperbolic manifold, $(\tilde{M},\tilde g)$, as our background manifold. Let $\tilde M=\tilde M_0\cup \tilde M_\infty$, where the metric on $\tilde M_0$ agrees with the metric on $M_0$.

Recall that the metric on $M_\infty$ is $g=dr^2+f(r,\theta)^2d\theta^2$, we define the function \begin{align*}
    \tilde f(r,\theta)=(1-a(r))  \sinh (  c_1r)+a(r)f(r,\theta)
\end{align*}
when $r\in[A,\infty)$, where $c_1>0$ is a constant that will be chosen later. Meanwhile, $a\in C^\infty([A,\infty))$, $a(r)\equiv1$ when $r\leq A+1/2$, and $a(r)\equiv 0$ when $r>A+1$. 
We choose $0\leq a\leq1$, $a'\leq0$, and
\[
(a'')_+\leq C_a(1-a).
\]
This can be arranged by choosing $a$ concave near the part where it leaves $1$, and allowing $a''>0$ only where $1-a$ is bounded below.
Then, we let $\tilde M_\infty=[A,\infty)\times\partial M$ and
\begin{align*}
    \tilde g=dr^2+\tilde f(r,\theta)^2d\theta^2
\end{align*} on $\tilde M_\infty.$

We also define $\chi\in C^\infty_0(M)$ to be a compact cutoff function of $M_0$, such that $\chi\equiv1$ on $M_0$, and $\chi\equiv0$ on $\{(r,\theta)\in [A+1/4,\infty)\times\partial M\}\subset M_\infty.$ Since $\chi$ is supported on the region where $M$ and $\tilde M$ agree, we may also identify $\chi$ as a function on $\tilde M$.

In the rest of this subsection, we will construct $\tilde M$, such that it has negative Gaussian curvature, and $\textnormal{supp}([\chi,\Delta_{\tilde g}])$ is nontrapping in $\tilde M$.
\begin{figure}[h!]
    \centering
    \includegraphics[width=0.8\linewidth]{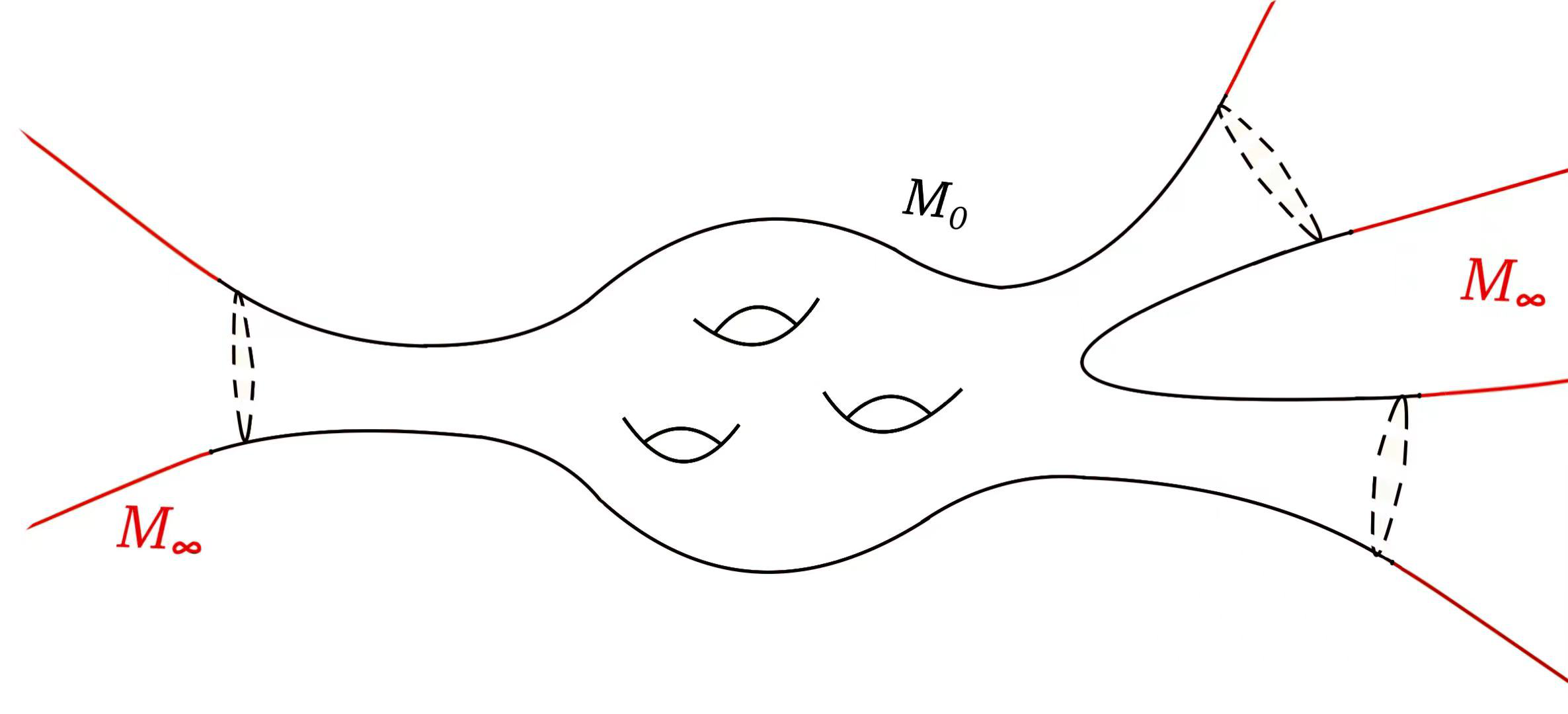}
    \caption{Asymptotically conic
 surface $M$}
    \label{fig:back}
\end{figure}

\begin{figure}[h!]
    \centering
    \includegraphics[width=0.8\linewidth]{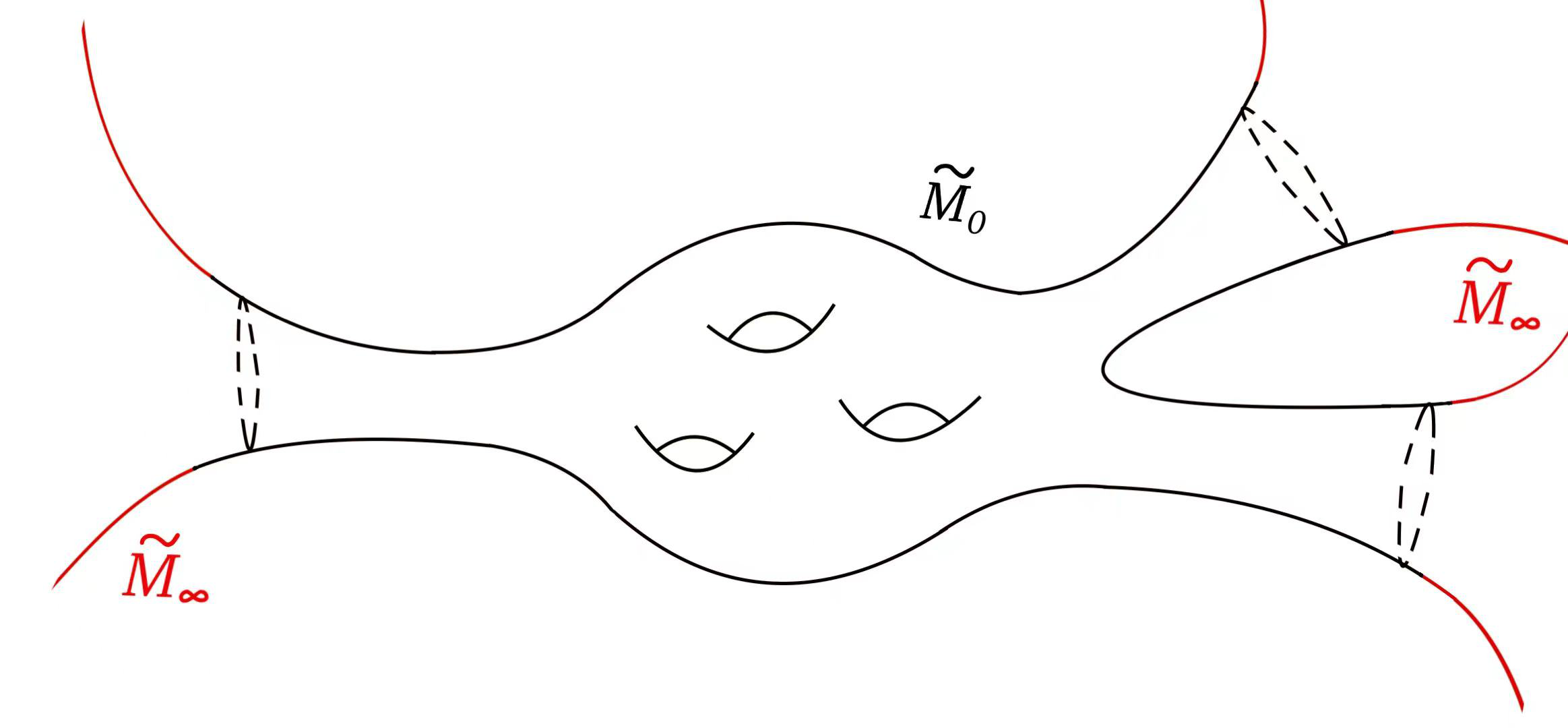}
    \caption{Asymptotically hyperbolic background surface $\tilde M$}
    \label{fig:M}
\end{figure}
\subsubsection{$\tilde M$ is even asymptotically hyperbolic with negative curvature}
Now, we compute the Gaussian curvature of $\tilde M$.
Recall that for a 2-dimensional manifold with metric \begin{align*}
    dr^2+h(r,\theta)^2d\theta^2,
\end{align*}
its Gaussian curvature is \begin{align*}
    -\frac{\partial_{rr}h}{h}.
\end{align*}
We assumed $M$ has negative curvature on $[A,A+1]$, then \begin{align*}
    \frac{\partial_{rr}f(r,\theta)}{f(r,\theta)}>0
\end{align*}on $[A,A+1]$.
We want to show \begin{align*}
    \frac{\partial_{rr}\tilde f(r,\theta)}{\tilde f(r,\theta)}=\frac{\partial_{rr}((1-a(r))  \sinh (  c_1r)+a(r)f(r,\theta))}{(1-a(r))  \sinh (  c_1r)+a(r)f(r,\theta)}>0.
\end{align*}
Note that \begin{align*}
    \partial_{rr}((1-a(r))&  \sinh (  c_1r)+a(r)f(r,\theta))\\
    =&a(r)\partial_{rr}f(r,\theta)+(1-a(r))    c_1^2\sinh(  c_1r)\\
    &+2a'(r)(\partial_rf(r,\theta)-    c_1\cosh(  c_1r))\\
    &+a''(r)(f(r,\theta)-  \sinh(  c_1r)).
\end{align*}
We first fix $a$, and then choose $c_1\gg1$, such that
$$\partial_rf(r,\theta)-c_1\cosh(c_1r)\leq0,$$
$$f(r,\theta)-\sinh(c_1r)\leq0$$
on $[A,A+1]$, and $c_1^2>C_a$. The term involving $a'$ is nonnegative. The term involving $a''$ is also nonnegative when $a''\leq0$; otherwise, it is bounded below by $-C_a(1-a)\sinh(c_1r)$. Thus, on $[A,A+1]$,
\[
\partial_{rr}\tilde f(r,\theta)
\geq a(r)\partial_{rr}f(r,\theta)
+(1-a(r))(c_1^2-C_a)\sinh(c_1r)>0.
\]
Outside the transition region, the metric either agrees with $g$ or has $\tilde f=\sinh(c_1r)$. Hence, $\tilde M$ has strictly negative curvature. The choices of $a$ and $c_1$ are fixed independently of $\la$ and $\delta$.
 
\subsubsection{$\textnormal{supp}([\chi,\Delta_{\tilde g}])$ is nontrapping in $\tilde M$} Now, we aim to show that for any geodesic in $\tilde M$ passing through $\textnormal{supp}([\chi,\Delta_{\tilde g}])$, at least one direction of the geodesic flow will diverge to infinity in physical space.

Let a geodesic on $M$ be $\gamma$. Let $\tilde\gamma$ be the geodesic on $\tilde M$, which agrees with $\gamma$ on $M_0$.  We parametrize $\tilde \gamma$  by $(r(t),\theta(t))$ for $t\in\mathbb{R}.$ If $\tilde \gamma$ can be defined by an element $(r,\theta,s,\xi)$ in $T^*\tilde M$, then we have $$\ddot{r}=\frac{\partial_r\tilde{f}(r,\theta)}{\tilde{f}^3(r,\theta)}\xi^2$$ along $\tilde \gamma.$

We first aim to show $\partial_r\tilde f>0$ when the metrics of $M$ and $\tilde M$ disagree. Compute
\begin{align}\label{partialtildef}
    \partial_r\tilde{f}(r,\theta)=a(r)\partial_rf(r,\theta)+(1-a(r))c_1\cosh(  c_1r)+a'(r)(f(r,\theta)-\sinh(  c_1r)).
\end{align}
We know $ a(r)\geq 0$ and $(1-a(r))    c_1\cosh(  c_1r)\geq 0$. Moreover, since  $f>0$ and $$\partial_rf^{-2}(r,\theta)=-\frac{2\partial_rf(r,\theta)}{f^3(r,\theta)}<0$$ by our assumption \eqref{assumption3} on $M$, we know $\partial_rf(r,\theta)>0.$ Finally, we have $a'(r)\leq 0$ for all $r$, and we may choose $c_1\gg1$ large enough, so that $f(r,\theta)-\sinh(c_1r)\leq 0$ for $r\in[A,A+1]$ to obtain $$a'(r)(f(r,\theta)-\sinh(c_1r))\geq0.$$ All the terms in \eqref{partialtildef} are nonnegative, and the first two have a positive sum. Thus, $\partial_r\tilde f>0$ throughout the end. Consequently, $\ddot r\geq0$, with strict inequality if the angular covector is nonzero. After reversing the direction of the geodesic if necessary, we may assume $\dot r\geq0$ at the point under consideration. If $\dot r=0$, the angular covector is nonzero and $\dot r$ becomes positive. For unit-speed radial geodesics, $\dot r=\pm1$. Therefore, in at least one direction, the geodesic escapes to infinity and cannot return after crossing a level of $r$ outward. Thus, $\operatorname{supp}([\chi,\Delta_{\tilde g}])$ is nontrapping.

For $r>A+1$, put $x=e^{-c_1r}$. Then
\[
c_1^2\tilde g=x^{-2}\left(dx^2+\frac{c_1^2}{4}(1-x^2)^2d\theta^2\right).
\]
Thus, $c_1^2\tilde g$ is even asymptotically hyperbolic in the usual normalization. We apply the results for even asymptotically hyperbolic surfaces after this fixed rescaling. This changes only the constants and the fixed frequency and time scales. We write $\tilde P=\sqrt{-\Delta_{\tilde g}}$.

\subsection{Proof of Theorem \ref{Strichartzthm}}

To prove Theorem \ref{Strichartzthm}, we recall the following lemmas.

\noindent (a) Lossless Strichartz and local smoothing estimates in the nontrapping region: Let $\chi\in C_0^{\infty}(M)$ with $\chi=1$ on a neighborhood of $\pi(K)$,
    \begin{equation}\label{eq:stri-nontrap-region}
    \|(1-\chi)e^{-it\Delta_g}u_0\|_{L^p_tL^q_x(M\times [0,1])}\leq C \|u_0\|_{L^2(M)}.
\end{equation}
(b) Lossless local smoothing in the nontrapping region on $M$: Fix $\beta\in C_0^{\infty}((1/2,2))$,  for $\chi\in C_0^{\infty}(M_\infty)$, we have 
\begin{equation}\label{eq:local-sm-nontrap-M}
   \|\chi e^{-it\Delta_g}\beta(\sqrt{-\Delta_g}/\la)u_0 \|_{L^2_{t,x}(M\times\mathbb R)}\leq C\la^{-1/2}\|u_0\|_{L^2(M)}.
\end{equation}
(c) Lossless local smoothing in the nontrapping region on $\tilde M$: If  we have $\chi\in C_0^{\infty}(\tilde M)$, such that $\chi=0$ on a neighborhood of the trapped set of $\tilde M,$ then \begin{equation}\label{eq:local-sm-nontrap}
   \|\chi e^{-it\Delta_{\tilde g}}\beta(\sqrt{-\Delta_{\tilde g}}/\la)u_0 \|_{L^2_{t,x}(\tilde M\times\mathbb R)}\leq C\la^{-1/2}\|u_0\|_{L^2(\tilde M)}.
\end{equation}

\noindent(d) Lossless Strichartz on asymptotically hyperbolic surface with negative curvature: For $(p,q)$ satisfying \eqref{keeltao}, we have
 \begin{equation}\label{eq:stri-unit-trap}
    \| e^{-it\Delta_{\tilde g}}\beta(\sqrt{-\Delta_{\tilde g}}/\la)u_0\|_{L^p_tL^q_x(\tilde M\times [0,1])}\leq C \|u_0\|_{L^2(\tilde M)}.
\end{equation}

The smoothing estimates above are used for $\la\gg1$ and on the whole time axis. We will also use, for any $\kappa\in C_0^\infty(M)$,
\begin{equation}\label{eq:local-sm-log-M}
\|\kappa e^{-it\Delta_g}\beta(P/\la)u_0\|_{L^2_{t,x}(M\times\mathbb R)}
\lesssim \la^{-1/2}(\log\la)^{1/2}\|u_0\|_{L^2(M)}.
\end{equation}
We first justify the resolvent input on $M$, so that the use of smoothing on the transition region does not require a further assumption.

By the main theorem of \cite{tao2026spectralgapsurfacesinfinite}, for any compact cutoff $\kappa$ on $\tilde M$,
\[
\|\kappa(-\Delta_{\tilde g}-(\la+i0)^2)^{-1}\kappa\|_{L^2(\tilde M)\to L^2(\tilde M)}
\lesssim\la^{-1}\log\la.
\]
The fixed rescaling above and the constant spectral shift in that theorem do not change this high-energy bound.

We transfer this bound to $M$ using \cite[Theorem 2.1]{DVgluing}. Put $h=\la^{-1}$ and consider $-h^2\Delta_g-1$ at energies in a fixed small neighborhood of zero. Choose the overlapping cutoff levels inside the collar where $g=\tilde g$. The interior model is $-h^2\Delta_{\tilde g}-1$, with cutoff resolvent bound
\[
a_1(h)\lesssim h^{-1}\log(1/h).
\]
For the exterior model, keep each end unchanged, cap it off inside the inner edge of the collar, and add a complex absorbing potential $-iW$, where $W\geq0$ is smooth. The function $W$ vanishes on the matching region and is positive before a geodesic reaches the part of the cap where radial convexity is no longer imposed. Every trajectory outside the absorbing region either reaches it or escapes to infinity. The nontrapping estimate and outgoing property therefore give
\[
a_0(h)\lesssim h^{-1};
\]
see \cite{VasyZworski2000} and the propagation argument with complex absorption in \cite[Sections 4 and 5.1]{DVgluing}.

The convexity condition \eqref{assumption3} holds on the original ends, and $\partial_r\tilde f>0$ holds on the background ends. Thus, after crossing a cutoff level outward, a geodesic cannot return. In particular the interior region is bicharacteristically convex in the background model. These are the convexity and no-return conditions in \cite[Section 2]{DVgluing}. The exterior model is outgoing on the overlap; for the interior model the outgoing condition is only needed on backward trajectories not meeting the core, and follows from propagation from its hyperbolic infinity. Thus, all the gluing is performed where the two metrics agree. The resulting bound is
\[
h^2a_0(h)^2a_1(h)\lesssim h^{-1}\log(1/h).
\]
Returning to the original frequency, we obtain
\begin{equation}\label{eq:res-log-M}
\|\kappa(-\Delta_g-(\la+i0)^2)^{-1}\kappa\|_{L^2(M)\to L^2(M)}
\lesssim \la^{-1}\log\la.
\end{equation}
Here we only use the boundary values at high positive energies, not continuation into a resonance-free strip. The same localized high-energy bound holds uniformly for the adjacent nonreal parameters by the limiting absorption argument.

We may now apply \cite[Theorem 1.2]{nontrappinglocalsmoothing} on $M$ to obtain the $O(\la^{-1})$ bound when both cutoffs are supported away from the trapped set. The same theorem and the background resolvent bound give this improvement on $\tilde M$. By \cite[Theorem 7.2 and Interpretation 2]{spectralbookdyatlov}, these bounds imply \eqref{eq:local-sm-nontrap-M}, \eqref{eq:local-sm-nontrap}, and \eqref{eq:local-sm-log-M}, with time interval $\mathbb R$ and a high-frequency cutoff. By the corresponding localized derivative estimate, we also have
\begin{equation}\label{eq:comm-sm}
\|[\Delta_g,\kappa]e^{-it\Delta_g}\beta(P/\la)u_0\|_{L^2_{t,x}(M\times\mathbb R)}
\lesssim \la^{1/2}\|u_0\|_{L^2(M)},
\end{equation}
if the derivatives of $\kappa$ are supported away from the trapped set; the same estimate holds on $\tilde M$. This follows by applying the localized smoothing estimate with one derivative, or by the argument of \cite[Lemma 2.6]{HSTZ}.
Finally, \eqref{eq:stri-nontrap-region} follows from \cite[Theorem 3.6]{BGH}, and \eqref{eq:stri-unit-trap} follows from \cite[Theorem 1.1]{HSTZ}.

In addition, we recall the Littlewood Paley estimate for manifolds with bounded geometry, \cite[Lemma 4.1]{HSTZ}.
\begin{lemma}\label{littlewood}
    Let 
    $\beta\in C_0^\infty (1/2, 2)$ with $\sum_{k=-\infty}^\infty \beta(s/2^k)=1$, and define $\beta_k(s)=\beta(s/2^k)$ for $k\geq1$, $\beta_0(s)= \sum_{k\le 0} \beta(s/2^k)$. If $(M,g)$ is a complete manifold of bounded geometry, we have for  $2\le q<\infty$
    \begin{equation}\label{little1}
        \|u\|_{L^q(M)}\lesssim \|Su\|_{L^q(M)} +\|u\|_{L^2(M)},
    \end{equation}
    where $Su=\left(\sum_{k\ge0} |\beta_k(\sqrt{-\Delta_g})u|^2\right)^{\frac12}$.
\end{lemma}
Now, we use \eqref{eq:stri-nontrap-region}--\eqref{little1} to obtain Theorem \ref{Strichartzthm}.

\begin{proof}[Proof of Theorem~\ref{Strichartzthm}]
By \eqref{little1}, it suffices to prove uniform estimates for data spectrally supported in a fixed dyadic annulus with $\la\gg1$. The bounded-frequency part follows from Sobolev estimates and unitarity. For this proof, we take $\beta$ to be an enlarged cutoff equal to one on this annulus, so that $u_0=\beta(P/\la)u_0$. The Littlewood--Paley summation still uses the original partition.

Notice that $M_\infty$ is assumed to be nontrapping.
By \eqref{eq:stri-nontrap-region}, it suffices to show that for any $\chi$, such that $\chi\equiv 1$ on $M_0$ and $\chi\equiv 0$ for $r\geq A+1/4$ on $M_\infty$, we have
    \begin{equation*}
        \|\chi e^{-it\Delta_g}u_0\|_{L^p_tL^q_x(M\times [0,1])}\leq C \|u_0\|_{L^2(M)}.
    \end{equation*}
    
 Let $u(t)=e^{-it\Delta_g}u_0$.
Notice that $v=\chi u$ solves \begin{equation*}
        \begin{cases}
            (i\partial_t-\Delta_{g})v=-[\Delta_g,\chi]u\\
            v(0)=\chi u_0.
        \end{cases}
    \end{equation*}
Then, \begin{equation*}
        v=e^{-it\Delta_{g}}\chi u_0+i\int^{t}_0e^{-i(t-s)\Delta_{g}}[\Delta_g,\chi]u(s) ds.
    \end{equation*}
    Notice that $\chi$ is supported where $g=\tilde g$. Thus, $v$ also solves \begin{equation*}
        \begin{cases}
            (i\partial_t-\Delta_{\tilde g})v=-[\Delta_g,\chi]u\\
            v(0)=\chi u_0,
        \end{cases}
    \end{equation*}  on $\tilde M$ and \begin{equation*}
        v=e^{-it\Delta_{\tilde g}}\chi u_0+i\int^{t}_0e^{-i(t-s)\Delta_{\tilde g}}[\Delta_g,\chi]u(s) ds.
    \end{equation*}
    
Let $\tilde \beta \in C^\infty_0((1/4,4))$ which equals one in a neighborhood of the support of $\beta$.
We aim to show 
\begin{align}\label{tailpart}
    \tilde\beta(\sqrt{-\Delta_{\tilde g}}/\la) \chi
\beta(\sqrt{-\Delta_{g}}/\la)f=\chi
\beta(\sqrt{-\Delta_{g}}/\la)f+R_1f,
\end{align}
where $\|R_1f\|_{L^q(\tilde M)}\le C_N\la^{-N}\|f\|_{L^2(M)}$  for $q\ge2$. 

By Lemma 4.2 of \cite{HSTZ}, we have\begin{align*}
    &\tilde\beta(\sqrt{-\Delta_{g}}/\la) \chi\\
    &=(2\pi)^{-1}\int \vartheta(t)\lambda\widehat{\tilde\beta}(\lambda t)\cos(tP)dt\chi+(2\pi)^{-1}\int (1-\vartheta(t))\lambda\widehat{\tilde\beta}(\lambda t)\cos(tP)dt\chi\\
    &=B\chi+R\chi
\end{align*}
 and 
 \begin{align*}
    &\tilde\beta(\sqrt{-\Delta_{\tilde g}}/\la) \chi\\
    &=(2\pi)^{-1}\int \vartheta(t)\lambda\widehat{\tilde\beta}(\lambda t)\cos(t\tilde P)dt\chi+(2\pi)^{-1}\int (1-\vartheta(t))\lambda\widehat{\tilde\beta}(\lambda t)\cos(t\tilde P)dt\chi\\
    &=\tilde B\chi+\tilde R\chi.
\end{align*}
Here $\tilde\beta$ is extended evenly to $\mathbb R$. We define $\vartheta\in C_0^\infty(\mathbb R)$ to be a function supported on $[-1/4,1/4]$ and $\equiv 1$ on $[-1/8,1/8]$. By the finite propagation of the Hadamard parametrix, $B\chi$ and $\tilde B\chi$ are supported on an $1/4$-neighborhood of $\textnormal{supp}(\chi)$. The remainders $R$ and $\tilde R$ are $O(\la^{-N})$ from $L^2$ to every fixed Sobolev space, for any $N$, by integration by parts in $t$ and the spectral theorem.

Notice that the metric on $M_0\cup (M_\infty\cap\{(r,\theta)\in [A,\infty)\times\partial M, r\leq A+1/2\})$ and $\tilde M_0\cup (\tilde M_\infty\cap\{(r,\theta)\in [A,\infty)\times\partial M, r\leq A+1/2\})$ are equal to each other. Thus, $B\chi$ agrees with $\tilde B\chi$ by the finite speed of propagation of Hadamard parametrix. By \cite[(2.7)]{HSTZ}, we know
$$\tilde\beta(\sqrt{-\Delta_{ g}}/\la) \chi\beta(\sqrt{-\Delta_{g}}/\la)=\chi\beta(\sqrt{-\Delta_{g}}/\la)+R',$$ for $||R'||_{L^2\to L^q}\lesssim O(\lambda^{-N})$. Thus, 
\begin{align}
\begin{split}
    B\chi
\beta(\sqrt{-\Delta_{g}}/\la)f&=\tilde\beta(\sqrt{-\Delta_{ g}}/\la) \chi\beta(\sqrt{-\Delta_{g}}/\la)f-R\chi\beta(\sqrt{-\Delta_{g}}/\la)f\\
&=\chi
\beta(\sqrt{-\Delta_{g}}/\la)f-R\chi \beta(\sqrt{-\Delta_{g}}/\la)f+R'f.
\end{split}
\end{align} Therefore,
\begin{align*}
    &\tilde\beta(\sqrt{-\Delta_{\tilde g}}/\la) \chi
\beta(\sqrt{-\Delta_{g}}/\la)f\\
&= \tilde B \chi
\beta(\sqrt{-\Delta_{g}}/\la)f+\tilde R\chi
\beta(\sqrt{-\Delta_{g}}/\la)f\\
&=B\chi
\beta(\sqrt{-\Delta_{g}}/\la)f+\tilde R \chi
\beta(\sqrt{-\Delta_{g}}/\la)f\\
&=\chi
\beta(\sqrt{-\Delta_{g}}/\la)f+R_1f,
\end{align*}
where $\|R_1f\|_{L^q(\tilde M)}\le C_N\la^{-N}\|f\|_{L^2(M)}$  for $q\ge2$ and any $N\in \mathbb{N}$. 

    Applying \eqref{tailpart} to $e^{-it\Delta_g}u_0$ gives an $O(\la^{-N})$ error uniformly on $[0,1]$. Thus, it suffices to estimate 
    \begin{equation*}
        \left\|e^{-it\Delta_{\tilde g}}\tilde\beta(\tilde P/\la)\chi u_0+i\int^{t}_0e^{-i(t-s)\Delta_{\tilde g}}\tilde\beta(\sqrt{-\Delta_{\tilde g}}/\la)[\Delta_g,\chi]u(s) ds\right\|_{L^p_tL^q_x([0,1]\times \tilde M)}.
    \end{equation*}

    For the first term above, by \eqref{eq:stri-unit-trap},
\begin{align*}
        \left\|e^{-it\Delta_{\tilde g}}\tilde\beta(\tilde P/\la)\chi u_0\right\|_{L^p_tL^q_x([0,1],\tilde M)}\lesssim||\chi u_0||_{L^2(M)}.
    \end{align*}
For the second term, notice that $[\Delta_g,\chi]$ is supported on the nontrapping region of both $\tilde M$ and $M.$ By \cite[Theorem 1.1]{HSTZ}, \eqref{eq:local-sm-nontrap-M} and \eqref{eq:local-sm-nontrap}, 
     \begin{align*}
       &\left\|\int_0^1e^{-i(t-s)\Delta_{\tilde g}}\tilde \beta(\sqrt{-\Delta_{\tilde g}}/\la)[\chi,\Delta_g]u(s,\cdot)ds\right\|_{L^p_tL^q_x([0,1],\tilde M)}\\
       &\lesssim\left\|\int_0^1e^{is\Delta_{\tilde g}}\tilde \beta(\sqrt{-\Delta_{\tilde g}}/\la)[\chi,\Delta_g]u(s,\cdot)ds\right\|_{L^2(\tilde M)}\\
       &\lesssim\lambda^{-1/2}||[\chi,\Delta_g]u(s,\cdot)||_{L^2_{s,x}([0,1],\tilde M)}\\
       &\lesssim||u_0||_{L^2(M)}.
    \end{align*}
   
    The same estimate holds for $$\int_0^te^{-i(t-s)\Delta_{\tilde g}}\tilde \beta(\sqrt{-\Delta_{\tilde g}}/\la)[\chi,\Delta_g]u(s,\cdot)ds$$ by the Christ--Kiselev Lemma.
    Finally, we apply Lemma~\ref{littlewood} and complete the proof of Theorem \ref{Strichartzthm}.
\end{proof}
\begin{remark}
    \textnormal{In \cite{HSTZ}, we proved the Strichartz estimate on even asymptotically hyperbolic surfaces with negative curvature. Now, we may use the same proof as Theorem~\ref{Strichartzthm} to prove the Strichartz estimate on asymptotically hyperbolic surfaces without assumptions on evenness and curvature on $M_\infty$. In this remark, we assume $M=M_0\cup M_\infty$ is an asymptotically hyperbolic surface satisfying \eqref{assumption1} and \eqref{assumption3}, and assume $M_0$ has negative sectional curvature.}
    
    \textnormal{We first use \cite[Theorem 1.2]{bouclet}, which gives the exterior Strichartz estimate with a cutoff equal to one on a sufficiently large compact set. To move the cutoff closer to the trapped set, we also use the high-frequency local smoothing estimate on a compact nontrapping region.}

\textnormal{For this estimate, the background construction is unchanged. For the exterior model, we cap off each end smoothly, keeping the original metric near the matching collar and choosing the radial coefficient with positive derivative away from the cap center. This gives a nontrapping asymptotically hyperbolic surface. We apply \cite[Theorems 1.1 and 1.2]{WangAH} to the Laplacian of this model to obtain a polynomial resolvent bound and the outgoing property without an evenness assumption. The gluing argument above therefore gives a polynomial cutoff resolvent bound on $M$; no logarithmic bound on the original $M$ is needed here. The sufficiently far-out $O(\la^{-1})$ estimate from \cite[Theorem 1.1]{Cardoso2002UniformEO}, followed by \cite[Theorem 1.3]{nontrappinglocalsmoothing} and propagation in the ends, gives the no-loss smoothing estimate on any fixed compact region disjoint from the trapped set. Applying the local argument of \cite[Theorem 3.6]{BGH} to the compact difference of the two cutoffs gives the following lemma.}
    \begin{lemma}
        Let $\chi\in C_0^{\infty}(M)$ with $\chi=1$ on a neighborhood the trapped set of $M$. For $(p,q)$ satisfying \eqref{keeltao},
    \begin{equation}\label{eq:hyperbolicstri-nontrap}
    \|(1-\chi)e^{-it\Delta_g}u_0\|_{L^p_tL^q_x(M\times [0,1])}\leq C \|u_0\|_{L^2(M)}.
\end{equation}
    \end{lemma}
    \textnormal{Then, we may use \eqref{assumption3} and the curvature assumption on $M_0$ to construct a background surface of $M$, which is even asymptotically hyperbolic with negative curvature and agrees with $M$ on $M_0$. Finally, we duplicate the proof of Theorem~\ref{Strichartzthm} to obtain the following.}
    \begin{theorem}
         Assuming $M$ as above, for $(p,q)$ satisfying \eqref{keeltao}, \begin{align}\label{eq:hyperbolicthmstrichart}
    \|e^{-it\Delta_g}u_0\|_{L^p_tL^q_x(M\times [0,1])}\leq C \|u_0\|_{L^2(M)}.
    \end{align}
    \end{theorem}
\end{remark}

\section{Spectral Projection Estimates}

We first prove Theorem~\ref{specthm} for $6<q<\infty$ in Subsections 3.1--3.3. Theorem~\ref{specthm-full} is proved in Subsection 3.4. Throughout this section, the ends are asymptotically Euclidean as in \eqref{eq:aemetric}.

Let $\rho\in {\mathcal S}(\mathbb R)$ satisfy $\rho(0)=1$ and have Fourier transform vanishing outside of $[-1,1]$. By a fixed rescaling, we may also assume
\[
|\rho(s)|\geq\tfrac12,\qquad |s|\leq3.
\]
Let $\varepsilon>0$
be sufficiently small. We choose $\chi_\infty\in C^\infty(M)$ such
that, on each end,
\[
\chi_\infty=0\quad\text{for }r\leq\varepsilon^{-1},
\qquad
\chi_\infty=1\quad\text{for }r\geq2\varepsilon^{-1}.
\]
In addition, let $\chi_0\in C_0^\infty(M)$ such that $\chi_0=1$ on a neighborhood of $M_0$ and $\chi_0$ is supported in $M_0$ together with the part of the collar where the metric agrees with the background metric constructed in Section 2. All the spatial cutoffs are real valued.
Finally, define
\[
\chi_c:=1-\chi_0-\chi_\infty.
\]

Let $\beta$ be the Littlewood-Paley function as before, and let
$f_\la:=\beta(\sqrt{-\Delta_g}/\la)f.$
It suffices first to prove that for $\la\gg1$ and $0<\delta\leq(\log\la)^{-1}$, we have
\begin{equation}\label{3.3}
\| \chi_0 \rho((\la \delta)^{-1}(-\Delta_g-\la^2))f_\la\|_{L^q(M)}\lesssim 
    \la^{\mu(q)}\delta^{{\frac{1}{2}}}\|f\|_{L^2(M)},
\end{equation}
\begin{equation}\label{3.5}
\| \chi_c \rho((\la \delta)^{-1}(-\Delta_g-\la^2))f_\la\|_{L^q(M)}\lesssim 
    \la^{\mu(q)}\delta^{{\frac{1}{2}}}\|f\|_{L^2(M)},
\end{equation}
and
\begin{equation}\label{3.4}
\| \chi_\infty \rho((\la \delta)^{-1}(-\Delta_g-\la^2))f_\la\|_{L^q(M)}\lesssim\la^{\mu(q)}\delta^{\frac12} \|f\|_{L^2(M)}.
\end{equation}

 \subsection{Estimate on support of $\chi_0$}
Define $\tilde M$ to be an asymptotically hyperbolic manifold with strictly negative curvature as in the previous section, such that $M$ and $\tilde M$ agree on a neighborhood of $\textnormal{supp}(\chi_0)$. We define $u:=e^{-it\Delta_g}f_\la$ and $v:=\chi_0e^{-it\Delta_g}f_\la.$ Then, $v$ solves the Cauchy problem,
\begin{equation}\label{6}
\begin{cases}
(i\partial_t-\Delta_g)v=[\chi_0,\Delta_g]u,\\
v|_{t=0}=\chi_0f_\la.
\end{cases}
\end{equation}
Since the metrics agree on a neighborhood of $\operatorname{supp}\chi_0$, $v$ solves the same problem with $\Delta_g$ on the left replaced by $\Delta_{\tilde g}$. Thus,
\begin{equation}\label{71}
v=e^{-it\Delta_{\tilde g}}(\chi_0f_\la)
+i\int_0^t e^{-i(t-s)\Delta_{\tilde g}}[\Delta_g,\chi_0]u(s,\cdot)\,ds.
\end{equation}
Choose $\alpha\in C_0^\infty((-1,1))$ satisfying $\sum_j\alpha(t-j)=1$. Let
\[
L=\log\la,\qquad a=\la/L,\qquad \alpha_j(t)=\alpha(at-j),
\qquad I_j=[t_j^-,t_j^+]=[(j-1)/a,(j+1)/a].
\]
We apply the time cutoff to the whole solution and set $v_j=\alpha_jv$. Then
\[
(i\partial_t-\Delta_{\tilde g})v_j=iF_{j,v}-F_{j,w},
\qquad F_{j,v}=\alpha_j'\chi_0u,\quad
F_{j,w}=\alpha_j[\Delta_g,\chi_0]u.
\]
For $\sigma=v,w$, define
\[
K_{j,\sigma}(t)=\int_{t_j^-}^t e^{is\Delta_{\tilde g}}F_{j,\sigma}(s)\,ds.
\]
Since $v_j(t_j^-)=0$, Duhamel's formula gives
\begin{equation}\label{7}
v_j(t)=e^{-it\Delta_{\tilde g}}\bigl(K_{j,v}(t)+iK_{j,w}(t)\bigr).
\end{equation}
By the inverse Fourier transform,
\begin{equation}\label{2.3chi0}
\chi_0\rho((\la\delta)^{-1}(-\Delta_g-\la^2))f_\la
=(2\pi)^{-1}\la\delta\sum_j\int_{I_j}
e^{-it\la^2}\widehat\rho(\la\delta t)v_j(t)\,dt.
\end{equation}

We recall the spectral projection theorem in \cite[Theorem 1.2]{HSTZ}.
\begin{lemma}\label{hyperbolicspec}
Let $\tilde M$ be an even asymptotically hyperbolic surface with negative curvature. Then for $\la\gg1$, we have the uniform bounds
\begin{equation}\label{ii.5}
\|\mathbf{1}_{[\la,\la+\delta]}(\tilde P)f\|_{L^q(\tilde M)}
\leq C_q\la^{\mu(q)}\delta^{1/2}\|f\|_{L^2(\tilde M)},
\qquad 6\leq q<\infty,\quad 0<\delta\leq1.
\end{equation}
\end{lemma}

Put $H_\la=\Delta_{\tilde g}+\la^2+ia$. For $\sigma=v,w$, define
\begin{align}\label{eq:R-core}
R'_{j,\sigma,\la}f
&=\la\delta\int_{I_j}e^{-itH_\la}
\frac{d}{dt}\bigl(e^{-at}\widehat\rho(\la\delta t)\bigr)K_{j,\sigma}(t)\,dt,\\
S_{j,\sigma,\la}f
&=\la\delta\int_{I_j}e^{-it\la^2}\widehat\rho(\la\delta t)F_{j,\sigma}(t)\,dt.\notag
\end{align}
The two Duhamel terms need not vanish separately at $t_j^+$. However, by \eqref{7},
\[
K_{j,v}(t)+iK_{j,w}(t)=0,\qquad t=t_j^-,t_j^+.
\]
Thus, the boundary terms cancel when we integrate the full expression in \eqref{7} by parts, and
\begin{equation}\label{eq:ibp-core}
\begin{split}
\la\delta\int_{I_j}e^{-it\la^2}\widehat\rho(\la\delta t)v_j(t)\,dt
={}&-iH_\la^{-1}\bigl(R'_{j,v,\la}f+S_{j,v,\la}f\bigr)\\
&+H_\la^{-1}\bigl(R'_{j,w,\la}f+S_{j,w,\la}f\bigr).
\end{split}
\end{equation}
Choose $\chi_1\in C_0^\infty(\tilde M)$ with $\chi_1=1$ on a neighborhood of $\operatorname{supp}\chi_0$. Since $v_j=\chi_1v_j$ and at most $O((\delta L)^{-1})$ intervals meet the support of $\widehat\rho(\la\delta t)$, the Cauchy--Schwarz inequality shows that \eqref{3.3} follows from
\begin{multline}\label{3.14q}
\sum_{\sigma=v,w}\left(\sum_j\|\chi_1H_\la^{-1}R'_{j,\sigma,\la}f\|_{L^q(\tilde M)}^2\right)^{1/2}
+\sum_{\sigma=v,w}\left(\sum_j\|\chi_1H_\la^{-1}S_{j,\sigma,\la}f\|_{L^q(\tilde M)}^2\right)^{1/2}\\
\lesssim \la^{\mu(q)}\delta L^{1/2}\|f\|_{L^2(M)}.
\end{multline}

We also state an analog of \cite[Lemma 2.5]{HSTZ}.
\begin{lemma}\label{global1}
Let $M$ and $\tilde M$ be as above. Let $\beta$ be the Littlewood Paley function, with $\beta_k(s)=\beta(s/2^k)$ for $k\geq1$ and $\beta_0(s)=\sum_{k\leq0}\beta(s/2^k)$. Suppose $f\in L^2(M)$ and $\la\gg1$. Let $\kappa\in C_0^\infty(M)$ have its derivatives supported in the region where the two metrics agree, away from both trapped sets. Let $\bigcup_j I_j=\mathbb R$, where the intervals have finite overlap and length $\lesssim1$. For $a_j\in C_0^\infty(I_j)$ with $|a_j|\lesssim1$, if $|\ln(2^k/\la)|>10$, then
\begin{multline}\label{eq:guess-2}
\left(\sum_j\left\|\int_{I_j}e^{is\Delta_{\tilde g}}\beta_k(\tilde P)
[\Delta_g,\kappa]\beta(P/\la)a_j(s)e^{-is\Delta_g}f\,ds\right\|_{L^2(\tilde M)}^2\right)^{1/2}\\
\lesssim_N\max\{2^k,\la\}^{-N}\|f\|_{L^2(M)}.
\end{multline}
If $|\ln(2^k/\la)|\leq10$, then
\begin{equation}\label{eq:guess-22}
\left(\sum_j\left\|\int_{I_j}e^{is\Delta_{\tilde g}}\beta_k(\tilde P)
[\Delta_g,\kappa]\beta(P/\la)a_j(s)e^{-is\Delta_g}f\,ds\right\|_{L^2(\tilde M)}^2\right)^{1/2}
\lesssim\|f\|_{L^2(M)}.
\end{equation}
The constants do not depend on derivatives of $a_j$. The same conclusions hold with $\tilde M=M$.
\end{lemma}
\begin{proof}
Extend $\beta$ evenly and use the same smooth time cutoff $\vartheta$ as in the proof of \eqref{tailpart}. Write
\[
\beta(P/\la)=(2\pi)^{-1}\int\vartheta(t)\la\widehat\beta(\la t)\cos(tP)\,dt+R,
\qquad
R=(2\pi)^{-1}\int(1-\vartheta(t))\la\widehat\beta(\la t)\cos(tP)\,dt.
\]
Here $\vartheta$ is chosen with support small enough for the metrics to agree on the required neighborhood of $\operatorname{supp}[\Delta_g,\kappa]$. The finite propagation argument is the same as in \eqref{tailpart}. For example, if $\ln(2^{k_1}/\la)<-10$, the remainder satisfies
\begin{equation}\label{goal2}
\begin{split}
&\sum_j\int_{I_j}\|\beta_{k_1}(\tilde P)[\Delta_g,\kappa]R\tilde\beta(P/\la)
a_j(s)e^{-is\Delta_g}f\|_{L^2(\tilde M)}^2\,ds\\
&\qquad\lesssim\la\|R\tilde\beta(P/\la)f\|_{L^2(M)}^2
\lesssim_N\la^{-N}\|f\|_{L^2(M)}^2.
\end{split}
\end{equation}
We used \eqref{eq:comm-sm}, finite overlap, and the fact that $R$ commutes with the original propagator. The spectral multiplier of $R\tilde\beta(P/\la)$ is rapidly decreasing in $\la$. For the short-time pieces, the off-frequency pseudodifferential composition has arbitrarily negative order. Its kernel is localized near $\operatorname{supp}[\Delta_g,\kappa]$, so the same local smoothing argument applies to each localized derivative. This proves \eqref{eq:guess-2}, also when $2^k\gg\la$; the output remainder is treated in the same way. This is the off-frequency argument in \cite[Lemma 2.5]{HSTZ} with the two propagators kept on their respective manifolds.

For \eqref{eq:guess-22}, choose $\psi$ equal to one on the commutator support and supported away from the background trapped set. The adjoint of \eqref{eq:local-sm-nontrap} bounds the time integral by
\[
C(2^k)^{-1/2}\|[\Delta_g,\kappa]\beta(P/\la)a_j(s)e^{-is\Delta_g}f\|_{L^2(I_j\times M)}.
\]
Squaring, summing, and using \eqref{eq:comm-sm} gives $(2^k)^{-1/2}\la^{1/2}\|f\|_2\lesssim\|f\|_2$. None of these estimates differentiates $a_j$.
\end{proof}

We apply this lemma with $\kappa=\chi_0$. By truncating each time weight to $[t_j^-,t]$ and approximating by smooth weights, the same bounds hold for partial integrals. The truncation endpoint can be chosen separately for each $j$. Summing the finitely many near-frequency terms and the rapidly decreasing off-frequency terms gives
\begin{equation}\label{eq:global-stri-diag}
\left(\sum_j\sup_{t\in I_j}\left\|\int_{t_j^-}^t
e^{is\Delta_{\tilde g}}\alpha_j(s)[\Delta_g,\chi_0]e^{-is\Delta_g}f_\la\,ds
\right\|_{L^2(\tilde M)}^2\right)^{1/2}
\lesssim\|f\|_{L^2(M)}.
\end{equation}

To prove the bounds for the ``$R'$-terms'' in \eqref{3.14q}, we state the following:
\begin{equation}\label{chi03.15}
\|H_\la^{-1}h\|_{L^q(\tilde M)}
\lesssim\la^{\mu(q)-1}L^{1/2}\|h\|_{L^2(\tilde M)}.
\end{equation}
Near frequency $\la$, this follows from Lemma~\ref{hyperbolicspec}, the Cauchy--Schwarz inequality and $L^2$ orthogonality. Away from this frequency range, the two-dimensional Sobolev estimate gives a smaller bound.

Notice that $\delta L\leq1$, so
\[
\int_{I_j}e^{at}\left|\frac{d}{dt}\bigl(e^{-at}\widehat\rho(\la\delta t)\bigr)\right|\,dt=O(1).
\]
Thus, by Minkowski's integral inequality,
\[
\|R'_{j,w,\la}f\|_{L^2(\tilde M)}
\lesssim\la\delta\sup_{t\in I_j}\|K_{j,w}(t)\|_{L^2(\tilde M)}.
\]
By \eqref{eq:global-stri-diag} and \eqref{chi03.15},
\[
\left(\sum_j\|\chi_1H_\la^{-1}R'_{j,w,\la}f\|_{L^q(\tilde M)}^2\right)^{1/2}
\lesssim\la^{\mu(q)-1}L^{1/2}\la\delta\|f\|_{L^2(M)}
=\la^{\mu(q)}\delta L^{1/2}\|f\|_{L^2(M)}.
\]

For the time cutoff terms, unitarity and $|\alpha_j'|\lesssim a$ give
\[
\sup_{t\in I_j}\|K_{j,v}(t)\|_{L^2(\tilde M)}
\lesssim a^{1/2}\|\chi_0u\|_{L^2(I_j\times M)}.
\]
The same Cauchy--Schwarz estimate applies to $S_{j,v,\la}$. Hence, by \eqref{eq:local-sm-log-M},
\begin{equation}\label{eq:core-time-terms}
\begin{split}
&\left(\sum_j\|R'_{j,v,\la}f\|_2^2\right)^{1/2}
+\left(\sum_j\|S_{j,v,\la}f\|_2^2\right)^{1/2}\\
&\qquad\lesssim\la\delta a^{1/2}\|\chi_0u\|_{L^2(\mathbb R\times M)}
\lesssim\la\delta\|f\|_2.
\end{split}
\end{equation}
Applying \eqref{chi03.15} proves the required bound for these two terms.

To estimate the remaining ``$S$-term'', we use the following two-sided localized resolvent estimate, \cite[Proposition 2.8]{HSTZ}.
\begin{proposition}\label{keyb}
Let $\tilde M$ be an even asymptotically hyperbolic surface with negative curvature, $\chi_1\in C_0^\infty(\tilde M)$, and $\tilde\chi_1\in C_0^\infty(\tilde M)$ supported away from the trapped set. Then, for $6\leq q<\infty$ and $\la\gg1$,
\begin{equation}\label{prop3.7}
\|\chi_1(\Delta_{\tilde g}+\la^2+i\la/\log\la)^{-1}(\tilde\chi_1h)\|_{L^q(\tilde M)}
\lesssim\la^{\mu(q)-1}\|h\|_{L^2(\tilde M)}.
\end{equation}
\end{proposition}
Since $\chi_0=1$ near $M_0$, $S_{j,w,\la}f$ is supported away from both trapped sets. By \eqref{eq:comm-sm},
\[
\left(\sum_j\|S_{j,w,\la}f\|_{L^2(\tilde M)}^2\right)^{1/2}
\lesssim\la\delta a^{-1/2}\|[\Delta_g,\chi_0]u\|_{L^2(\mathbb R\times M)}
\lesssim\la\delta L^{1/2}\|f\|_2.
\]
Thus, \eqref{prop3.7} gives the remaining bound in \eqref{3.14q}. Combining the four estimates proves \eqref{3.3}.

\subsection{Estimate on support of $\chi_c$}
Let $u_j=\alpha_j(t)\chi_ce^{-it\Delta_g}f_\la$. Then
\[
(i\partial_t-\Delta_g)u_j=v_j+w_j,
\]
where
\begin{equation}\label{eq:def-ujvj}
v_j=i\alpha_j'(t)\chi_cu,\qquad
w_j=-\alpha_j(t)[\Delta_g,\chi_c]u.
\end{equation}
We use the same intervals and notation $L=\log\la$, $a=\la/L$ as above, and now put $H_\la=\Delta_g+\la^2+ia$. Define
\[
K_{j,v}(t)=\int_{t_j^-}^t e^{is\Delta_g}\alpha_j'(s)\chi_cu(s)\,ds,
\qquad
K_{j,w}(t)=\int_{t_j^-}^t e^{is\Delta_g}\alpha_j(s)[\Delta_g,\chi_c]u(s)\,ds.
\]
Then $u_j=e^{-it\Delta_g}(K_{j,v}+iK_{j,w})$. The complete solution $u_j$ vanishes at both endpoints of $I_j$. As in \eqref{eq:ibp-core}, integration by parts gives
\begin{equation}\label{eq:ibp-exterior}
\begin{split}
\la\delta\int_{I_j}e^{-it\la^2}\widehat\rho(\la\delta t)u_j(t)\,dt
={}&-iH_\la^{-1}(R'_{j,v,\la}f+S_{j,v,\la}f)\\
&+H_\la^{-1}(R'_{j,w,\la}f+S_{j,w,\la}f),
\end{split}
\end{equation}
where
\begin{align*}
R'_{j,v,\la}f&=\la\delta\int_{I_j}e^{-itH_\la}
\frac{d}{dt}(e^{-at}\widehat\rho(\la\delta t))K_{j,v}(t)\,dt,\\
S_{j,v,\la}f&=\la\delta\int_{I_j}e^{-it\la^2}\widehat\rho(\la\delta t)
\alpha_j'(t)\chi_ce^{-it\Delta_g}f_\la\,dt,\\
R'_{j,w,\la}f&=\la\delta\int_{I_j}e^{-itH_\la}
\frac{d}{dt}(e^{-at}\widehat\rho(\la\delta t))K_{j,w}(t)\,dt,\\
S_{j,w,\la}f&=\la\delta\int_{I_j}e^{-it\la^2}\alpha_j(t)\widehat\rho(\la\delta t)
[\Delta_g,\chi_c]e^{-it\Delta_g}f_\la\,dt.
\end{align*}
Here the boundary terms of the two sources cancel in their combination, not separately.

Let us fix $\chi_1\in C_0^\infty(M)$ such that $\chi_1=1$ on a neighborhood of $\operatorname{supp}\chi_c$. Then $u_j=\chi_1u_j$. Since there are $O((\delta\log\la)^{-1})$ intervals meeting the time support of $\widehat\rho(\la\delta t)$, by the Cauchy--Schwarz inequality, we would obtain \eqref{3.5} if we could show
\begin{multline}\label{3.13}
\left(\sum_j\|\chi_1H_\la^{-1}R'_{j,v,\la}f\|_{L^q(M)}^2\right)^{1/2}
+\left(\sum_j\|\chi_1H_\la^{-1}S_{j,v,\la}f\|_{L^q(M)}^2\right)^{1/2}\\
\lesssim\la^{\mu(q)}\delta L^{1/2}\|f\|_{L^2(M)},
\end{multline}
and
\begin{multline}\label{3.14}
\left(\sum_j\|\chi_1H_\la^{-1}R'_{j,w,\la}f\|_{L^q(M)}^2\right)^{1/2}
+\left(\sum_j\|\chi_1H_\la^{-1}S_{j,w,\la}f\|_{L^q(M)}^2\right)^{1/2}\\
\lesssim\la^{\mu(q)}\delta L^{1/2}\|f\|_{L^2(M)}.
\end{multline}

We first recall the sharp spectral projection estimates in logarithmic scale on manifolds with nonpositive curvature and bounded geometry from \cite[Theorem 1.6]{HSTZ}. Notice that $M$ satisfies the assumptions in Lemma~\ref{bgsp}.
\begin{lemma}\label{bgsp}
Suppose that $(M,g)$ is a complete surface of uniformly bounded geometry and nonpositive curvature. Then, for $\la\gg1$,
\begin{equation}\label{i.6}
\|\mathbf{1}_{[\la,\la+(\log\la)^{-1}]}(P)\|_{L^2(M)\to L^q(M)}
\lesssim\la^{\mu(q)}(\log\la)^{-1/2},\qquad 6<q<\infty.
\end{equation}
\end{lemma}
On $M$, we also have the following lemma. Its proof is the same off-frequency argument as in Lemma~\ref{global1}; the near-frequency bound follows from Cauchy--Schwarz in time and \eqref{eq:local-sm-nontrap-M}.
\begin{lemma}\label{global2}
Suppose $f\in L^2(M)$ and $\la\gg1$. 
Assume $\chi \in C_0^{\infty}(M)$ with $\chi=0$ on a neighborhood of $M_{0}$.
    Let $\bigcup_{j} I_j=\mathbb R$ where the intervals $I_j$ have finite overlap and length $\lesssim 1$, then for any $a_j\in C_0^\infty(I_j)$ with $|a_j|\lesssim 1$ and $|\ln(2^k/\la)|> 10$, we have
\begin{multline}\label{eq:guess-3}
  \left( \sum_{j}  \left\| \int_{I_j} e^{is\Delta_{g}} \beta_{k}(\sqrt{-\Delta_{g}})\chi \beta(\sqrt{-\Delta_g}/\la) a_j(s)e^{-is\Delta_g}f ds \right\|^2_{L^2(M)}\right)^{\frac12}\\
  \lesssim_N     \max\{\la,2^{k}    \}^{-N}\|f\|_{L^2(M)}.
\end{multline} 
In addition, if $|\ln(2^k/\la)|\leq 10$, then
\begin{equation}\label{eq:guess-33}
    \left( \sum_{j} \left\| \int_{I_j} e^{is\Delta_{g}} \beta_{k}(\sqrt{-\Delta_{ g}})\chi\beta(\sqrt{-\Delta_g}/\la) a_j(s)e^{-is\Delta_g}f ds \right\|^2_{L^2(M)}\right)^{\frac12}\lesssim \la^{-\frac12}\|f\|_{L^2(M)}.
\end{equation}
\end{lemma}

To prove the bounds for the ``$R'$-terms'' in \eqref{3.13} and \eqref{3.14}, we state the following:
\begin{equation}\label{3.15}
\|H_\la^{-1}h\|_{L^q(M)}
\lesssim\la^{\mu(q)-1}L^{1/2}\|h\|_{L^2(M)}.
\end{equation}
Near frequency $\la$, this follows from Lemma~\ref{bgsp}, the Cauchy--Schwarz inequality and $L^2$ orthogonality. The remaining frequencies are bounded by the two-dimensional Sobolev estimate. All these bounds are used for a fixed $6<q<\infty$.

By Minkowski's integral inequality and the time-weight bound used above,
\begin{equation}\label{eq:Rv-pointwise}
\|R'_{j,v,\la}f\|_2\lesssim\la\delta\sup_{t\in I_j}\|K_{j,v}(t)\|_2.
\end{equation}
Since $|\alpha_j'|\lesssim a$ and $|I_j|\lesssim a^{-1}$,
\[
\sup_{t\in I_j}\|K_{j,v}(t)\|_2
\lesssim a^{1/2}\|\chi_cu\|_{L^2(I_j\times M)}.
\]
Thus, by finite overlap and \eqref{eq:local-sm-nontrap-M},
\begin{equation}\label{3.16}
\begin{split}
\left(\sum_j\|R'_{j,v,\la}f\|_2^2\right)^{1/2}
&\lesssim\la\delta a^{1/2}\|\chi_cu\|_{L^2(\mathbb R\times M)}\\
&\lesssim\la\delta L^{-1/2}\|f\|_2.
\end{split}
\end{equation}
By \eqref{3.15}, we obtain
\[
\left(\sum_j\|\chi_1H_\la^{-1}R'_{j,v,\la}f\|_{L^q(M)}^2\right)^{1/2}
\lesssim\la^{\mu(q)}\delta\|f\|_2,
\]
which is better than the required bound in \eqref{3.13}.

For the commutator term, apply Lemma~\ref{global1} with both propagators on $M$ and $\kappa=\chi_c$. Its derivatives are supported away from the trapped set. The partial-integral argument used in \eqref{eq:global-stri-diag} gives
\[
\left(\sum_j\sup_{t\in I_j}\|K_{j,w}(t)\|_2^2\right)^{1/2}\lesssim\|f\|_2.
\]
Consequently,
\[
\left(\sum_j\|\chi_1H_\la^{-1}R'_{j,w,\la}f\|_{L^q(M)}^2\right)^{1/2}
\lesssim\la^{\mu(q)-1}L^{1/2}\la\delta\|f\|_2
=\la^{\mu(q)}\delta L^{1/2}\|f\|_2.
\]
Similarly, H\"older's inequality and \eqref{eq:local-sm-nontrap-M} give
\begin{align*}
\left(\sum_j\|\chi_1H_\la^{-1}S_{j,v,\la}f\|_{L^q(M)}^2\right)^{1/2}
&\lesssim\la^{\mu(q)-1}L^{1/2}\la\delta a^{1/2}
\|\chi_cu\|_{L^2(\mathbb R\times M)}\\
&\lesssim\la^{\mu(q)}\delta\|f\|_2.
\end{align*}
This proves \eqref{3.13} and the first part of \eqref{3.14}.

To estimate the second term in \eqref{3.14}, we need the following two-sided localized resolvent estimate.
\begin{proposition}\label{keyeuc}
Let $(M,g)$ be an asymptotically Euclidean surface as above with nonpositive curvature, $\chi_1\in C_0^\infty(M)$, and $\tilde\chi_1\in C_0^\infty(M_\infty)$ supported away from the trapped set. Then, for $6<q<\infty$ and $\la\gg1$,
\begin{equation}\label{3.15a}
\|\chi_1(\Delta_g+\la^2+i\la/\log\la)^{-1}(\tilde\chi_1h)\|_{L^q(M)}
\lesssim\la^{\mu(q)-1}\|h\|_{L^2(M)}.
\end{equation}
\end{proposition}
First, by H\"older's inequality and \eqref{eq:comm-sm},
\begin{equation}\label{eq:Sw-exterior}
\begin{split}
\left(\sum_j\|S_{j,w,\la}f\|_{L^2(M)}^2\right)^{1/2}
&\lesssim\la\delta a^{-1/2}\|[\Delta_g,\chi_c]u\|_{L^2(\mathbb R\times M)}\\
&\lesssim\la\delta L^{1/2}\|f\|_2.
\end{split}
\end{equation}
Therefore, \eqref{3.15a} proves the second term of \eqref{3.14}. Now, we prove \eqref{3.15a}.
\begin{proof}[Proof of Proposition~\ref{keyeuc}]
We follow the proof of \cite[Proposition 2.8]{HSTZ}.
In this proof, $n=2$. By a fixed rescaling, we may assume the injectivity radius is greater than $10$. To prove Proposition~\ref{keyeuc}, we show the following version of \eqref{3.15a},
\begin{equation}\label{main}
    \left\|\chi_1\left(\Delta_g+(\la+i(\log\la)^{-1})^2\right)^{-1}\tilde\chi_1\right\|_{L^2(M)\to L^q(M)}\lesssim \la^{\mu(q)-1}.
\end{equation}
The same proof is uniform if $(\log\la)^{-1}$ in \eqref{main} is replaced by $c(\log\la)^{-1}$, with $c$ in a fixed compact subinterval of $(0,\infty)$. Writing $\sqrt{\la^2+i\la/\log\la}=\nu+i\eta$, where $\nu\sim\la$ and $\eta\sim(\log\la)^{-1}$, therefore gives \eqref{3.15a}.
We fix $\beta \in C^\infty_0((1/4,4))$ with $\beta= 1$ in (1/2,2). By Sobolev estimates, we have \begin{equation}\label{p2.5}
    \left\|\chi_1\left(\Delta_g+(\la+i(\log\la)^{-1})^2\right)^{-1}(I-\beta(P/\la))\tilde\chi_1\right\|_{L^2(M)\to L^q(M)}\lesssim \la^{\mu(q)-1}.
\end{equation} So, it suffices to show 
\begin{equation}\label{eq:res-near}
    \left\|\chi_1\left(\Delta_g+(\la+i(\log\la)^{-1})^2\right)^{-1}\beta(P/\la)\tilde\chi_1\right\|_{L^2(M)\to L^q(M)}\lesssim \la^{\mu(q)-1}.
\end{equation}

Notice that we have the following identity
\begin{equation}\label{resolvent}
    \left(\Delta_g+(\la+i\delta)^2\right)^{-1}=\frac1{i(\la+i\delta)}\int_0^\infty e^{it\la-t\delta}\cos (t\sqrt{-\Delta_g})\,dt.
\end{equation}

We recall the analog of \cite[Lemma 2.9]{HSTZ}. The proof of Lemma~\ref{loc} is identical to the proof of \cite[Lemma 2.9]{HSTZ} using the local smoothing on $M$, \eqref{eq:local-sm-nontrap-M}.
\begin{lemma}\label{loc}  {Let} $\mu\in [\la/2,2\la]$, $\la\gg1$, $\delta\in (0,1/2)$ and $\tilde \chi\in C^\infty_0(M_\infty)$. {Then}
\begin{equation}\label{3.19}
\| \mathbf{1}_{[\mu,\mu+\delta)}(\sqrt{-\Delta_g}) (\tilde \chi h)\|_{L^2( M)}\lesssim \delta^{1/2}\|h\|_{L^2( M)}.
\end{equation}
\end{lemma}
We also need the following lemma.
\begin{lemma}
    There exist zero-order pseudodifferential operators $A_\pm$ {with compactly supported Schwartz kernel} such that 
{
\begin{equation}\label{eq:lem2.8}
\beta(P/\la)\tilde\chi_1=A_++ A_-+R,
\end{equation} 
}with
$ \|R\|_{L^2(M)\to L^2(M)}\lesssim\la^{-1}.
$
{In local coordinates, $A_{\pm}$ is of the form
\begin{equation}
    A_{\pm} u(x) = (2\pi)^{-n}\la^n\int\int e^{i\la\langle x-y,\xi \rangle} A_{\pm}(x,y,\xi)u(y) dyd\xi,\quad A_{\pm}(x,y,\xi)\text{ smooth and compactly supported}.
\end{equation}
}
For some small $\delta_1>0$ and all $(x,y,\xi)\in \{(x,y,\xi):\textnormal{dist}((x,y,\xi),\textnormal{supp}\,A_+)\leq\delta_1\}$, if $(x(t),\xi(t))=\Phi_t(x,\xi)$, we have 
\begin{equation}\label{6a}
   \textnormal{dist}( x(t), \textnormal{supp}\,\chi_1)\ge 1  \,\,\,\text{ for }\,\,\,t \ge C,
\end{equation}
for some sufficiently large constant C. Similarly, for all $(x, y,\xi)\in \textnormal{supp}\, A_-(x,y,\xi)$,  we have 
\begin{equation}\label{7a}
   \textnormal{dist}( x(t), \textnormal{supp}\,\chi_1)\ge 1 \,\,\,\text{ for }\,\,\,t\le -C.
\end{equation}
\end{lemma}
\begin{proof}
If we extend $\beta$ to be an even function, then we can write $\beta(P/\la)=B+C$ where $ \|C\|_{L^2\to L^2}\lesssim_N \la^{-N}$, and $B$ is a pseudodifferential operator with principal symbol $\beta(p(x,\xi))$,
with $p(x,\xi)$ here now being the principal symbol of $P$. 

Next, choose $\psi\in C_0^\infty{(M)}$ with $\psi=1$ in a neighborhood of the support of $\tilde\chi_1$  and $\psi=0$ on $M_0$. Without loss of generality, we may assume both $\psi$ and $\tilde \chi_1$ are supported in a sufficiently small neighborhood of some fixed point $y_0$. Then, in normal coordinates around $y_0$, if $B(x,y)$ is the {Schwartz} kernel of $B$, 
we have $B(x,y)\tilde\chi_1(y)=\psi(x)B(x,y)\tilde \chi_1(y)+O(\la^{-N})$. {Since $B$ has principal symbol $\beta(p(x,\xi))$},
\begin{equation}\label{2.95a}
    B(x,y)\tilde\chi_1(y)=(2\pi)^{-n}\la^n\int e^{i\la\langle x-y,\xi \rangle} \psi(x){\beta(p(x,\xi))}\tilde\chi_1(y)d\xi +R(x,y),
\end{equation}
where $R$ is a lower order {pseudodifferential} operator which satisfies $ \|R\|_{L^2(M)\to L^2(M)}=O(\la^{-1})
$.

{Let $S=\{(x,\xi)\in S^*M: x\in \textnormal{supp}\,\psi \}$. 
Since $\psi(x)=0$ on $M_0$ and $\pi(\Gamma_+\cap \Gamma_-)\subset M_0$, the two sets $\Gamma_+\cap S$ and $\Gamma_-\cap S$ are disjoint. Since $\Gamma_\pm$ are closed and $S$ is compact, $\Gamma_+\cap S$ and $\Gamma_-\cap S$ are both compact. Hence, $\textnormal{dist}(\Gamma_+\cap S,\Gamma_-\cap S)>0.$
there exists $\phi_{\pm}\in C_0^{\infty}(S^*M)$ subordinate to the open cover $S\subset (U\setminus \Gamma_-)\cup (U\setminus \Gamma_+)$ where $U$ is a small neighbourhood of $S$, such that
\begin{equation}
    \phi_+(x,\xi) +\phi_-(x,\xi) =1,\quad (x,\xi)\in S,
\end{equation}
with $\textnormal{supp}\,\phi_+ \cap \Gamma_-=\emptyset$ 
 and $\textnormal{supp}\,\phi_- \cap \Gamma_+=\emptyset$. If we define the operators $A_\pm$ by 
\begin{equation}\label{10}
   A_\pm f(x)=(2\pi)^{-n}\la^n\int e^{i\la\langle x-y,\xi \rangle} \phi_{\pm}(x,\xi/|\xi|_g)\psi(x)\beta(p(x,\xi))\tilde\chi_1(y) f(y)d\xi dy,
\end{equation}
then we obtain \eqref{eq:lem2.8}. Moreover, by our assumption, \eqref{assumption2}, we have the convex geodesic flow on $M_\infty$. So,  $(x,\xi)\in \textnormal{supp} \, \phi_{\pm}$ satisfies \eqref{6a} and \eqref{7a} for sufficiently large $C$, respectively.
}
\end{proof}
Meanwhile, notice that
\begin{equation}\label{bbound}
    \|A_\pm\|_{L^p(M)\to L^p(M)}=O(1),\,\,\,\forall\,\,\,1\le p\le \infty.
\end{equation}

To prove \eqref{eq:res-near}, it suffices to show 
\begin{equation}\label{01}
    \left\|\chi_1\left(\Delta_g+(\la+i(\log\la)^{-1})^2\right)^{-1}\circ A\right\|_{L^2(M)\to L^q(M)}\lesssim \la^{\mu(q)-1}, \quad \text{ for } A=A_+, A_-.
\end{equation}
{{
Meanwhile, by \eqref{3.15} along with the fact that $\|R\|_{L^2(M)\to L^2(M)}\lesssim\la^{-1}$, we have
\begin{equation}\label{n12b}
\begin{aligned}
       \left\|\left(\Delta_g+(\la+i(\log\la)^{-1})^2\right)^{-1} R \,(h)\right\|_{ L^q(M)}&\lesssim \la^{\mu(q)-1}(\log\la)^{\frac12}\|R( h)\|_{L^2( M)}\\
       &\lesssim 
   \la^{\mu(q)-2}(\log\la)^{\frac12} \| h\|_{L^2( M)},
\end{aligned}
\end{equation}
which is better than the required estimate in \eqref{eq:res-near}.}}

Let us first prove \eqref{01} for $A=A_+$.  
We fix $\gamma\in C_0^\infty ((1/2, 2))$ satisfying $\sum_{j=-\infty}^\infty \gamma(s/2^j)=1$, and 
define 
\begin{equation}\label{tj}
T_j f=\frac1{i(\la+i(\log\la)^{-1})}\int_0^\infty \gamma(2^{-j}t)e^{it\la-t/\log\la}\cos (t\sqrt{-\Delta_g})f\,dt.
\end{equation}
Then, it suffices to estimate the $L^2(M)\to L^q(M)$ bounds for the $T_j$ operators.  The symbol of $T_j$ is 
\begin{equation}\label{symbol}
    T_j(\tau)=\frac1{i(\la+i(\log\la)^{-1})}\int_0^\infty \gamma(2^{-j}t)e^{it\la-t/\log\la}\cos (t\tau)\,dt=O(\la^{-1}2^j(1+2^j|\tau-\la|)^{-N}).
\end{equation}


Note that $(\Delta_g+(\la+i(\log\la)^{-1})^2)^{-1}=\sum_{-\infty}^\infty T_j$.  To prove the proposition, we may separately consider the cases when
$2^j\lesssim  1$, $1\lesssim  2^j\lesssim \log \la$ and $\log\la \lesssim 2^j$.

\noindent(i) $2^j\le 10C$ for $C$ as in \eqref{7a}.

First, if $2^j\in[\la^{-1},1]$, we will show that for
$q>6$,
\begin{equation}\label{tjbound1}
    \|T_j\|_{L^2(M)\to L^q(M)}
    \lesssim
    \la^{\mu(q)-1}2^{j/2}.
\end{equation}

Notice that $M$ has nonpositive curvature. By the Hadamard parametrix
for $\cos(t\sqrt{-\Delta_g})$, if $\la^{-1}\leq2^j\leq1$, then the
kernel of $T_j$ satisfies
\begin{equation}
T_j(x,y)
=
\begin{cases}
\la^{\frac{n-3}{2}}
2^{-j\frac{n-1}{2}}
e^{i\la d_g(x,y)}a_\la(x,y),
&
d_g(x,y)\in[2^{j-2},2^{j+2}],
\\
O\bigl(\la^{-1}2^{-j(n-1)}\bigr),
&
d_g(x,y)\leq2^{j-2},
\end{cases}
\end{equation}
where
\[
|\nabla_{x,y}^{\alpha}a_\la(x,y)|
\leq
C_\alpha d_g(x,y)^{-|\alpha|}.
\]
Additionally, by the finite propagation speed property of the wave
propagator,
\[
T_j(x,y)=0
\qquad\text{if}\qquad
d_g(x,y)\geq2^{j+2}.
\]

If
$d_g(x,y)\in[2^{j-2},2^{j+2}]$, the bound in
\eqref{tjbound1} follows from the oscillatory integral bounds of
H\"ormander \cite{HormanderIII} and Stein \cite{Stein}, combined
with a scaling argument. Indeed, putting $r=2^j$, the oscillatory
parameter after rescaling is $\la r$, and the standard
$L^2\to L^q$ oscillatory integral estimate gives
\begin{align*}
\|T_j\|_{L^2(M)\to L^q(M)}
&\lesssim
\la^{\frac{n-3}{2}}
r^{-\frac{n-1}{2}}
r^{\frac n2+\frac nq}
(\la r)^{-\frac nq}
\\
&=
\la^{\frac{n-3}{2}-\frac nq}r^{1/2}
\\
&=
\la^{\mu(q)-1}2^{j/2}.
\end{align*}

For $d_g(x,y)\leq2^{j-2}$, Young's inequality and the second
kernel bound give
\begin{align*}
\|T_j\|_{L^2(M)\to L^q(M)}
&\lesssim
\la^{-1}
2^{-j(n-1)}
2^{jn(\frac12+\frac1q)}
\\
&=
\la^{-1}
2^{-j\mu(q)}
2^{j/2}
\\
&\lesssim
\la^{\mu(q)-1}2^{j/2},
\end{align*}
where we used $\la2^j\geq1$ and $\mu(q)\geq0$.
Thus, \eqref{tjbound1} follows.

The finitely many terms with $1\leq2^j\leq10C$ can be handled by
the same local argument on the universal cover, summing the finitely many relevant deck transformations, with constants depending on $C$.

On the other hand, suppose $2^j\leq\la^{-1}$. Since we only need
to estimate the resolvent composed with $A$, recall that $A$ is
microlocalized to frequencies comparable to $\la$. Hence, if
$\tilde\beta\in C_0^\infty((1/4,4))$ equals one on the frequency
support of $A$, then
\[
A=\tilde\beta(P/\la)A+O(\la^{-N})_{L^2\to H^s},\qquad s\geq0\text{ fixed}.
\]
Also, by the spectral theorem and the definition of $T_j$,
\[
\left\|
\sum_{\{j:2^j\leq\la^{-1}\}}T_j
\right\|_{L^2(M)\to L^2(M)}
\lesssim
\la^{-2}.
\]
Therefore, by semiclassical Sobolev estimates,
\begin{align*}
\left\|
\sum_{\{j:2^j\leq\la^{-1}\}}T_j\circ A
\right\|_{L^2(M)\to L^q(M)}
&\lesssim
\la^{n(\frac12-\frac1q)}
\la^{-2}
\\
&=
\la^{\mu(q)-\frac32}
\\
&\lesssim
\la^{\mu(q)-1}.
\end{align*}
This proves the desired bound in case (i).

\noindent(ii) $2^j\ge c_0\log\la$ for some small enough $c_0$.

Let
$E_{\la,k}=\mathbf{1}_{[\la+k/\log\la, \la +(k+1)/\log\la)}(\sqrt{-\Delta_g}).$ 
Then by integration by parts in $t$-variable, the symbol of 
$$S_k=\frac1{i(\la+i(\log\la)^{-1})}E_{\la,k}\int_0^\infty e^{it\la-t/\log\la}\cos (tP)\,\sum_{2^j\ge c_0\log\la}\gamma(2^{-j}t)\,dt
$$
is $O\left(\la^{-1}\log\la (1+|k|)^{-N}\right )$. Thus, by the sharp spectral projection bound in Lemma~\ref{bgsp},
\begin{align*}
\|&\mathbf{1}_{[\la/2,2\la]}( \sqrt{-\Delta_g}) \, 
 \sum_{|k|\lesssim \la\log\la}(
 S_k\circ A h)\|_{L^q( M)}
\\
&\le \sum_{|k|\lesssim \la\log\la}
\|\mathbf{1}_{[\la/2,2\la]}( \sqrt{-\Delta_g}) \,(S_k\circ A h)\|_{L^q( M)}
\\
&\le  \la^{\mu(q)}(\log\la)^{-1/2}  \sum_{|k|\lesssim \la\log\la}
\| \mathbf{1}_{[\la/2,2\la]}( \sqrt{-\Delta_g}) \, (S_k\circ A h)\|_{L^2( M)}
\\
&\lesssim \la^{\mu(q)}(\log\la)^{-1/2}\sum_{|k|\lesssim \la\log\la} (1+|k|)^{-N} \la^{-1}\log\la  \| E_{\la,k} \circ (A h)\|_{L^2( M)}
\\
&\lesssim \la^{\mu(q)-1}  \|h\|_{L^2( M)},
\end{align*}
 using \eqref{3.19} with $\delta=(\log\la)^{-1}$, \eqref{10} and \eqref{bbound} in the last step.
The case when $|k|\gtrsim \la\log\la$ can be handled by Sobolev estimates.

\noindent(iii) $10C \le 2^j\le c_0\log\la$ for $C$ as in \eqref{7a}.

By duality, it suffices to show that
 the operator
\begin{equation}\label{20}
  T= \sum_{\{j: 10C\le 2^j\le c_0\log\la\}}\frac i{(\la-i(\log\la)^{-1})}\int_0^\infty \gamma(2^{-j}t)e^{-it\la-t/\log\la}   A^*\circ \cos (t\sqrt{-\Delta_g})\chi_1\,dt
\end{equation}
satisfies the same $L^{q'}(M)\to L^2(M)$ bound as in  \eqref{01}.

Since $(M,g)$ has nonpositive sectional curvature, we may use the 
Cartan-Hadamard theorem to lift the calculation up to the universal
cover of $(M,g)$ using the formula
from \cite[(3.6.4)]{SoggeHangzhou}
\begin{equation}\label{k.3}
    (\cos t\sqrt{-\Delta_{g}})(x,y)
=\sum_{\alpha\in \Gamma}
(\cos t\sqrt{-\Delta_{\mathring g}})(\mathring x,\alpha(\mathring y)).
\end{equation}
Here, $(\mathbb{R}^n,\mathring g)$ is the universal cover of 
$(M,g)$, with $\mathring g$ now being the Riemannian metric
on $\mathbb{R}^n$ obtained by pulling back the metric $g$
via the covering map.   Also, $\Gamma$ denotes the group of deck transformations, and $\mathring x, \mathring y \in D$ with $D\simeq M$ being a Dirichlet fundamental domain.

Notice that by the finite speed of propagation,  $\cos (t\sqrt{-\Delta_{\mathring g}})(\mathring x,\alpha(\mathring y))=0$ if $d_{\mathring g}(\mathring x,\alpha(\mathring y))>|t|$.
Thus, for each fixed $\mathring x$, since $(M,g)$ is of bounded geometry,  the number of deck transformations $\alpha$ such that $d_{\mathring g}(\mathring x,\alpha(\mathring y))\lesssim c_0\log\la$ is $O(\la^{C_Mc_0})$.

It suffices to show that for each fixed $\alpha$, we have
\begin{equation}\label{20a}
\int_0^\infty \gamma(2^{-j}t)e^{-it\la-t/\log\la}   (A^*\circ \cos (t\sqrt{-\Delta_{\mathring g}}))(\mathring x,\alpha(\mathring y))\chi_1(\mathring y)\,dt \lesssim_N \la^{-N}.
\end{equation}
Here, we slightly abuse the notation by identifying $\chi_1$ with a compactly supported function on the fundamental domain. We apply the Hadamard parametrix and use \eqref{6a} to prove \eqref{20a}. 

We may
use the Hadamard parametrix to express $\cos t\sqrt{-\Delta_{\mathring g}}$  in normal coordinates around some $\mathring x_0$
  as follows:
\begin{equation}\label{ak13}
\cos t\sqrt{-\Delta_{\mathring g}}(\mathring x,
 \mathring y)=
\sum_{\nu=0}^N
w_\nu(\mathring x,\mathring y)W_\nu(t, \mathring x, \mathring y)
+R_N(t, \mathring x, \mathring y),
\end{equation}
where $w_\nu \in C^\infty({\mathbb R}^{n}
\times {\mathbb R}^{n})$ for every $\nu$ and
\begin{equation}\label{k14}
W_0(t, \mathring x, \mathring y)=(2\pi)^{-n}
\int_{{\mathbb R}^{n}} e^{id_{\mathring g}( \mathring x,
  \mathring y)\xi_1} \cos t|\xi| \, d\xi,
\end{equation}
while for $\nu=1,2,\dots$, $W_\nu(t, \mathring x, \mathring y)$
is a finite linear combination of Fourier integrals 
of the form
\begin{equation}\label{k15}
\int_{{\mathbb R}^{n}}
e^{id_{\mathring g}( \mathring x, \mathring y)\xi_1}
e^{\pm it|\xi|}
\alpha_\nu(|\xi|)\, d\xi,
\, \, \text{with } \, \,
\alpha_\nu(\tau)=0, \, \text{for } \, \tau\le 1
\, \, \text{and } \, \, \partial^j_\tau \alpha_\nu(\tau)
\lesssim \tau^{-\nu-j}.
\end{equation}
In addition, for a given $N_0$ and a large enough $N$, we have
\begin{equation}\label{k16}
|\partial_t^j R_N(t, \mathring x, \mathring y)|\le C\exp(Ct),
\quad
0\le j\le N_0,
\end{equation}
for a fixed constant $C$. 
Finally, the coefficients 
$w_\nu( x, y)$ satisfy
\begin{equation}\label{k17}
0<w_0(\mathring x,\mathring y)\le 1,
\end{equation}
and
\begin{equation}\label{k19}
|\partial^\beta_x w_\nu( \mathring x, \mathring y)|
\le C\exp(Cr), \, \,
|\beta|, \nu\le N_0, \quad r=d_{\mathring g}( \mathring x,
 \mathring y),
\end{equation}
for some uniform constant $C$ depending only 
$ g$ and $N_0$. By the G\"unther comparison theorem, since the curvature of $M$ is bounded below by some negative constant, we also have the following bound for the distance function
\begin{equation}\label{k19a}
|\partial^\beta_{\mathring x,\mathring y} d_{\mathring g}( \mathring x,
 \mathring y)|
\le C\exp(Cr), \, \,\, \,\, \,\text{ for }
|\beta|\le N_0 \text{ and }  r=d_{\mathring g}( \mathring x, \mathring y)\geq1.
\end{equation}
We refer \cite[\S 1.1, \S 3.6]{SoggeHangzhou} for the above facts.

We focus on the $\nu=0$ term. The higher order terms can be treated similarly and satisfy better bounds. The error term involving $R_N$ satisfies the desired bound by using  \eqref{k16} and an integration by parts argument in the $t$ variable.

We use the explicit compactly supported amplitudes of $A_\pm$ in \eqref{10}. The lower-order remainder in the decomposition \eqref{eq:lem2.8} was already estimated in \eqref{n12b}.

Let $$H_p=\frac{\partial p}{\partial \xi}\frac{\partial }{\partial x}-\frac{\partial p}{\partial x}\frac{\partial }{\partial \xi}$$ denote the Hamilton vector field associated with the principal symbol $p(x,\xi)$ of $\sqrt{-\Delta_g}$.
Let $\Phi_t=e^{tH_p}: {T^*M\setminus \{0\}\to T^*M\setminus \{0\}}$ denote the geodesic flow on the cotangent bundle generated by $H_p$.

The kernel of the $\nu=0$ term in \eqref{20a} is 
\begin{equation}\label{mk}
     \begin{aligned}
        K(\mathring x,\alpha(\mathring y))=(2\pi)^{-2n}\la^{2n}&\iiiint \gamma(2^{-j}t)e^{-it\la-t/\log\la} e^{i\la\langle \mathring x-\mathring z,\xi \rangle} \phi_+(\mathring z,\xi/|\xi|_g)\beta(p(\mathring z,\xi))\\        &\,\,\cdot\tilde\chi_1(\mathring x)
        \cdot\psi(\mathring z)\chi_1(\mathring y)w_0(\alpha(\mathring y),\mathring z) e^{-i\la d_{\mathring g}(\alpha(\mathring y), \mathring z)\eta_1} \cos (t\la|\eta|) d\mathring zd\xi d\eta dt    \end{aligned}\end{equation}
 Up to a factor $1/2$, we can replace $\cos(t\la|\eta|)$ by $e^{it\la|\eta|}$ since the term involving $e^{-it\la|\eta|}$ is rapidly decreasing through integration by part in the $t$-variable. A similar integration by parts argument in $\eta, \mathring z$ variables also shows that we may assume  $d_{\mathring g}(\mathring z, \alpha(\mathring y))\ge 4C$ and $\eta_1\in [1/4,4]$.

 We claim that by choosing $c_0$ small enough, we have   $$\textnormal{dist}\left((\mathring z, -\eta_1\nabla_{\mathring z}d_{\mathring g}(\mathring z,\alpha(\mathring y))), \textnormal{supp} \phi_+(\mathring z,\xi/|\xi|_g)\right)\ge \la^{-\frac14}$$ This implies that the kernel in \eqref{mk} is $O(\la^{-N})$ by integration by parts in $\mathring z$ or $\xi$. We shall prove this by contradiction. Assume we have 
 \begin{equation}\label{s1}
     \textnormal{dist}\left((\mathring z,-\eta_1\nabla_{\mathring z}d_{\mathring g}(\mathring z,
 \alpha(\mathring y))), \textnormal{supp} \phi_+(\mathring z,\xi/|\xi|_g)\right)\le \la^{-\frac14}. \end{equation}
 Recall that the distance function satisfies 
 \begin{equation}\label{k9a}
 e^{-tH_p}(\mathring z,\nabla_{\mathring z}d_{\mathring g}(\mathring z,
 \alpha(\mathring y)))=(\alpha(\mathring y),\xi_0),
 \end{equation}
 for some $\xi_0$ and $t=d_{\mathring g}(\mathring z,\alpha(\mathring y))$. After reversing and scaling the covector, the endpoint covector is $-\eta_1\xi_0$. Since the Hamilton flow $\Phi_{t}$ is an isomorphism and satisfies $\Phi_{t+s}=\Phi_{t}\circ \Phi_{s}$, and because  $(\mathbb{R}^n,\mathring g)$ has bounded geometry,  \eqref{s1}, \eqref{k9a} along with the fact that $\eta_1\in [1/4,4]$ imply that there  exist a point $(z,\xi)\in \textnormal{supp} \phi_+(z,\xi/|\xi|_g)$ such that  \begin{equation}\label{s2}     \textnormal{dist}\left(e^{tH_p}(z,\xi),\, (\alpha(\mathring y),-\eta_1\xi_0)\right)\le \la^{-\frac14}\la^{Kc_0}. \end{equation}for some $t\ge C$ and fixed constant $K$. Since we are assuming $\mathring y\in \textnormal{supp} \chi_1$, by choosing $c_0$ sufficiently small, this contradicts \eqref{6a} after projecting $x(t)$ and $\alpha(\mathring y)$ back to $M$.
This finishes the proof of \eqref{01} if $A=A_+$. 

If  $A=A_-$, we can use a similar argument to show that 
\begin{equation}\label{01a}
    \left\|\chi_1\left(\Delta_g+(\la-i(\log\la)^{-1})^2\right)^{-1}\circ A\right\|_{L^2(M)\to L^q(M)}\lesssim \la^{\mu(q)-1},
\end{equation}
as well as 
\begin{align}\label{b-}
\left\|\chi_1\left(\left(\Delta_g+(\la+i(\log\la)^{-1})^2\right)^{-1}-\left(\Delta_g+(\la-i(\log\la)^{-1})^2\right)^{-1}\right)\circ A\right\|_{L^2(M)\to L^q(M)}
\lesssim \la^{\mu(q)-1}.
\end{align}
These two inequalities yield \eqref{01} with $A=A_-$.

By repeating the above arguments, \eqref{01a} is a consequence of 
\begin{equation}\label{20b}
\int_0^\infty \gamma(2^{-j}t)e^{it\la-t/\log\la}   (A^*\circ \cos (t\sqrt{-\Delta_{\mathring g}}))(\mathring x,\alpha(\mathring y))\chi_1(\mathring y)\,dt \lesssim_N \la^{-N}.
\end{equation}

On the other hand, put
\[
R_\pm=(\Delta_g+(\la\pm i(\log\la)^{-1})^2)^{-1},\qquad
E_{\la,k}=\mathbf{1}_{[\la+k/L,\la+(k+1)/L)}(P),
\quad L=\log\la.
\]
For frequencies in $[\la/2,2\la]$, the multiplier of $E_{\la,k}(R_+-R_-)$ is $O(\la^{-1}L(1+|k|)^{-2})$. We keep $A$ on the right. Choose $\psi_1\in C_0^\infty(M_\infty)$ equal to one on a neighborhood of the support of the spatial cutoff $\psi$ in \eqref{10}. Then $A=\psi_1A$. By Lemma~\ref{loc} and \eqref{bbound},
\[
\|E_{\la,k}Ah\|_2\leq\|E_{\la,k}\psi_1\|_{2\to2}\|Ah\|_2
\lesssim L^{-1/2}\|h\|_2.
\]
Applying Lemma~\ref{bgsp} to each spectral window before summing, we obtain
\begin{align*}
&\|\chi_1\mathbf{1}_{[\la/2,2\la]}(P)(R_+-R_-)Ah\|_{L^q(M)}\\
&\quad\lesssim\la^{\mu(q)}L^{-1/2}
\sum_{|k|\lesssim\la L}\la^{-1}L(1+|k|)^{-2}\|E_{\la,k}Ah\|_2\\
&\quad\lesssim\la^{\mu(q)-1}\sum_k(1+|k|)^{-2}\|h\|_2
\lesssim\la^{\mu(q)-1}\|h\|_2.
\end{align*}
The frequencies outside $[\la/2,2\la]$ are handled by the two-dimensional Sobolev estimate as in \eqref{p2.5}. This proves \eqref{b-}.
\end{proof}

\subsection{Estimate on support of $\chi_\infty$}

To establish \eqref{3.4}, we state the following estimate.

\begin{proposition}\label{new}
Let $M$ be a surface with asymptotically Euclidean ends as in \eqref{eq:aemetric}, and assume $\la\gg1$. There exists an $\varepsilon>0$ such that, for $6\leq q<\infty$ and $\mu(q)$ as in \eqref{ii.2}, if
\[
\operatorname{supp}\chi_\infty
\subset\{r>\varepsilon^{-1}\},
\]
then
\begin{align}
\|\chi_\infty dE_P(\la)\chi_\infty\|_{L^{q'}(M)\to L^q(M)}
\lesssim \la^{2\mu(q)}.
\end{align}
\end{proposition}

The corresponding estimate for $n\geq3$ follows from \cite[Proposition 8.2]{Guillarmou2010RestrictionAS}. We now explain why its high-energy proof also applies when $n=2$.

The proof of \cite[Proposition 8.2]{Guillarmou2010RestrictionAS}
uses three high-energy ingredients. The first is the cutoff resolvent
estimate \cite[(8.10)]{Guillarmou2010RestrictionAS}, which follows
from \cite[Theorem 1.1]{Cardoso2002UniformEO}. This estimate is valid
in dimension two.

The second ingredient is the nontrapping restriction estimate used
to obtain \cite[(8.12)]{Guillarmou2010RestrictionAS}. For high
energy, the microlocal spectral measure estimates needed for this
restriction estimate are proved in
\cite[Section 7]{Guillarmou2010RestrictionAS}, more precisely in
Section 7D, where the estimates \cite[(1.9)]{Guillarmou2010RestrictionAS}
are established. Although the main theorem in
\cite{Guillarmou2010RestrictionAS} is stated for $n\geq3$, the
high-energy argument in Section 7D does not use the low-energy
construction. Moreover, when $n=2$, the even-dimensional
Stein--Tomas argument only requires the estimates in \cite[(1.9)]
{Guillarmou2010RestrictionAS} for derivatives of order $0$ and $1$,
and the same high-energy argument gives these estimates.
More precisely, on each auxiliary nontrapping surface, let $\Delta$ denote its Laplacian and $d$ its geodesic distance. The high-energy construction gives a finite microlocal partition $\sum_iQ_i(\la)=I$, with the number of terms and their $L^2$ norms uniformly bounded, such that
\[
\left|\left(Q_i(\la)\partial_\sigma^j dE_{\sqrt{-\Delta}}(\sigma)Q_i(\la)^*\right)(x,y)\right|
\lesssim\sigma^{1-j}(1+\sigma d(x,y))^{-1/2+j},\qquad j=0,1,
\]
uniformly for $\sigma/\la\in(1-\eta,1+\eta)$ and $\la\gg1$, where $\eta>0$ is fixed and small. The microlocal supports are chosen with fixed margins, so $Q_i(\la)$ is held fixed while $\sigma$ varies; the derivatives act only on the spectral measure.

We may therefore apply \cite[Theorem 1]{xichen}, which gives the
Stein--Tomas restriction estimate from these microlocal spectral
measure estimates. Thus, on each auxiliary nontrapping
asymptotically Euclidean surface appearing in the proof of
\cite[Proposition 8.2]{Guillarmou2010RestrictionAS}, with $\Delta$
denoting its Laplacian,
\begin{align}
\|dE_{\sqrt{-\Delta}}(\la)\|_{L^{q'}\to L^q}
\lesssim
\la^{1-\frac4q}
=
\la^{2\mu(q)},
\qquad 6\leq q<\infty.
\end{align}
Using the factorization of the spectral measure and the $TT^*$
argument as in
\cite[Section 8]{Guillarmou2010RestrictionAS}, this gives
\cite[(8.12)]{Guillarmou2010RestrictionAS} in dimension two.

The third ingredient is
\cite[Lemma 8.3]{Guillarmou2010RestrictionAS}. Its proof uses a
Calder\'on--Zygmund estimate and requires $q'>1$, which holds since
$q<\infty$. We may therefore follow the remainder of the proof of
\cite[Proposition 8.2]{Guillarmou2010RestrictionAS} to obtain
Proposition~\ref{new} when $n=2$.

Notice that the restriction $n\geq3$ in
\cite[Theorem 1.3]{Guillarmou2010RestrictionAS} is related to the
low-energy analysis. As explained in
\cite[Section 9]{Guillarmou2010RestrictionAS}, the difficulty in
dimension two comes from the construction at the zero face. This
issue does not occur here since we only consider $\la\gg1$.

We now prove \eqref{3.4}. Let
\[
m(\tau)
=
\rho((\la\delta)^{-1}(\tau^2-\la^2))
\beta(\tau/\la).
\]
Then
\[
\chi_\infty
\rho((\la\delta)^{-1}(-\Delta_g-\la^2))f_\la
=
\chi_\infty m(P)f.
\]
Let
\[
T=\chi_\infty m(P).
\]
By the spectral theorem,
\[
TT^*
=
\chi_\infty |m|^2(P)\chi_\infty
=
\int_0^\infty
|m(\tau)|^2
\chi_\infty dE_P(\tau)\chi_\infty\,d\tau.
\]
Therefore, by Proposition~\ref{new},
\begin{align*}
\|T\|_{L^2(M)\to L^q(M)}^2
&=
\|TT^*\|_{L^{q'}(M)\to L^q(M)}\\
&\lesssim
\int_0^\infty
|m(\tau)|^2\tau^{2\mu(q)}\,d\tau\\
&\lesssim
\la^{2\mu(q)}
\int_0^\infty
\left|
\rho((\la\delta)^{-1}(\tau^2-\la^2))
\right|^2
|\beta(\tau/\la)|^2\,d\tau.
\end{align*}
Since $\tau\sim\la$ on the support of $\beta(\tau/\la)$, the change
of variables
\[
s=(\la\delta)^{-1}(\tau^2-\la^2)
\]
gives
\[
d\tau=\frac{\la\delta}{2\tau}\,ds\lesssim\delta\,ds.
\]
Hence,
\[
\|\chi_\infty m(P)\|_{L^2(M)\to L^q(M)}
\lesssim
\la^{\mu(q)}\delta^{1/2},
\]
which proves \eqref{3.4} for $6\leq q<\infty$.

To pass from the smooth multiplier to the spectral projection, notice that for $\tau\in[\la,\la+\delta]$,
\[
0\leq\frac{\tau^2-\la^2}{\la\delta}\leq2+\frac{\delta}{\la}\leq3.
\]
The Littlewood--Paley partition and its support condition imply $\beta(1)=1$. Thus, our choice of $\rho$ gives $|m(\tau)|\geq c>0$ on this interval, uniformly for large $\la$ and $0<\delta\leq1$. Define
\[
b_{\la,\delta}(\tau)=
\begin{cases}
m(\tau)^{-1},&\tau\in[\la,\la+\delta],\\
0,&\text{otherwise}.
\end{cases}
\]
Then $\|b_{\la,\delta}(P)\|_{2\to2}\lesssim1$ and
\[
\mathbf{1}_{[\la,\la+\delta]}(P)=m(P)b_{\la,\delta}(P).
\]
Combining \eqref{3.3}, \eqref{3.5} and \eqref{3.4}, we obtain
\[
\|\mathbf{1}_{[\la,\la+\delta]}(P)f\|_{L^q(M)}
\lesssim\la^{\mu(q)}\delta^{1/2}\|f\|_{L^2(M)},
\qquad 6<q<\infty,\quad 0<\delta\leq(\log\la)^{-1}.
\]
For $(\log\la)^{-1}<\delta\leq1$, split the spectral window into disjoint intervals of length at most $(\log(2\la))^{-1}$. There are $O(\delta\log\la)$ such intervals. Applying the small-window estimate to each interval, then using the Cauchy--Schwarz inequality and $L^2$ orthogonality, gives the same bound with $\delta^{1/2}$. This proves Theorem~\ref{specthm}.

\subsection{Proof of Theorem~\ref{specthm-full}}

Choose $\chi_\infty$ as in Theorem~\ref{specthm-full} and put
\[
\chi_0=1-\chi_\infty,\qquad \chi_c=0.
\]
The negative curvature assumption holds on a neighborhood of the whole support of $\chi_0$, including the support of its derivatives. On each end, let $R$ be the outer endpoint of $\operatorname{supp}\chi_0$ in the radial variable. Since $\chi_\infty$ depends only on $r$ and is nondecreasing, we may choose
\[
R<R_0<R_1
\]
so that $[R,R_1]\times\mathbb S^1$ is still contained in this negative curvature neighborhood. We carry out the construction in Subsection 2.1 with $a=1$ for $r\leq R_0$ and $a=0$ for $r\geq R_1$, retaining $a'\leq0$ and $(a'')_+\leq C_a(1-a)$. Any fixed positive width $R_1-R_0$ is enough: we first choose $a$ and then take $c_1$ sufficiently large as above. These choices may depend on the ends, but a single $c_1$ can be chosen for all of them, independently of $\la$ and $\delta$.

Thus, $\tilde M$ has negative curvature and its metric agrees with $g$ on a neighborhood of $\operatorname{supp}\chi_0$. The metric transition is entirely inside the negative curvature region. For $r\geq R_1$, we have $\tilde f=\sinh(c_1r)$, so no curvature sign of the original metric is used there. The background is even asymptotically hyperbolic after the same fixed rescaling as in Subsection 2.1. We choose $\chi_1=1$ on a neighborhood of $\operatorname{supp}\chi_0$ as in Subsection 3.1. The commutator $[\Delta_g,\chi_0]$ is supported in the ends, away from the trapped sets of both surfaces, by \eqref{assumption3} and the radial convexity of the background metric.

The proofs of \eqref{eq:local-sm-log-M} and \eqref{eq:comm-sm} apply to this background and these cutoffs. They use the negative curvature of the background and the convexity of the original ends, not a curvature sign on the remaining part of $M$. Thus, all the inputs on the original surface in Subsection 3.1 remain available.

First fix $6\leq q<\infty$ and $0<\delta\leq(\log\la)^{-1}$. Lemma~\ref{hyperbolicspec} and Proposition~\ref{keyb} include $q=6$; these are \cite[Theorem 1.2 and Proposition 2.8]{HSTZ} on the negatively curved background. In particular, if $L=\log\la$ and $H_\la=\Delta_{\tilde g}+\la^2+i\la/L$, the two resolvent bounds used at the endpoint are
\[
\|H_\la^{-1}\|_{L^2(\tilde M)\to L^6(\tilde M)}
\lesssim\la^{-5/6}L^{1/2},\qquad
\|\chi_1H_\la^{-1}\tilde\chi_1\|_{L^2(\tilde M)\to L^6(\tilde M)}
\lesssim\la^{-5/6},
\]
where $\tilde\chi_1=1$ near the commutator support and is supported away from the background trapped set. Consequently, the same four estimates in \eqref{3.14q} give
\[
\|(1-\chi_\infty)m(P)f\|_{L^q(M)}
\lesssim\la^{\mu(q)}\delta^{1/2}\|f\|_{L^2(M)},
\qquad 6\leq q<\infty,
\]
with $m$ as in Subsection 3.3. Proposition~\ref{new} and the $TT^*$ argument there also include $q=6$, and give the same bound for $\chi_\infty m(P)f$. No curvature sign is used in this exterior estimate. There is no $\chi_c$ term, so neither Lemma~\ref{bgsp} nor Proposition~\ref{keyeuc} is needed here.

Adding the two parts and using the same bounded multiplier $b_{\la,\delta}(P)$ as in Subsection 3.3 gives
\begin{equation}\label{eq:full-supercritical}
\|\mathbf{1}_{[\la,\la+\delta]}(P)f\|_{L^q(M)}
\lesssim\la^{\mu(q)}\delta^{1/2}\|f\|_{L^2(M)},
\qquad 6\leq q<\infty.
\end{equation}
As before, splitting a wider spectral window into disjoint smaller intervals and using $L^2$ orthogonality extends this estimate to $0<\delta\leq1$.

For $2<q<6$, we interpolate the estimate at $q=6$ with
\[
\|\mathbf{1}_{[\la,\la+\delta]}(P)\|_{L^2(M)\to L^2(M)}\leq1.
\]
Writing
\[
\frac1q=\frac{1-\theta}{2}+\frac\theta6,
\qquad \theta=3\left(\frac12-\frac1q\right)=6\mu(q),
\]
we obtain
\begin{align*}
\|\mathbf{1}_{[\la,\la+\delta]}(P)f\|_{L^q(M)}
&\lesssim\la^{\theta/6}\delta^{\theta/2}\|f\|_{L^2(M)}\\
&=\la^{\mu(q)}\delta^{3\mu(q)}\|f\|_{L^2(M)}\\
&\leq(\la\delta)^{\mu(q)}\|f\|_{L^2(M)},
\end{align*}
where we used $\delta\leq1$. This proves Theorem~\ref{specthm-full}.

\printbibliography

@article {SoggeHuangQuasimode2025,
    AUTHOR = {Huang, X. and Sogge, C. D.},
     TITLE = {Curvature and sharp growth rates of log-quasimodes on compact
              manifolds},
   JOURNAL = {Invent. Math.},
  FJOURNAL = {Inventiones Mathematicae},
    VOLUME = {239},
      YEAR = {2025},
    NUMBER = {3},
     PAGES = {947--1008},
      ISSN = {0020-9910,1432-1297},
   MRCLASS = {58J50 (35J05 35P15 35R01)},
  MRNUMBER = {4861130},
       DOI = {10.1007/s00222-025-01315-2},
}

@article {HSTZ,
    AUTHOR = {Huang, X. and Sogge, C. and Tao, Z. and Zhang, Z.},
     TITLE = {Lossless Strichartz and spectral projection estimates on unbounded manifolds},
   JOURNAL = {Geom. Funct. Anal.},
   volume={0 (0000) 1–93},
   year={2026}
}

@article{nontrappinglocalsmoothing,
     author = {Datchev, K. and Vasy, A.},
     title = {Propagation through trapped sets and semiclassical resolvent estimates},
     journal = {Annales de l'Institut Fourier},
     pages = {2347--2377},
     year = {2012},
     publisher = {Association des Annales de l'Institut Fourier},
     volume = {62},
     number = {6},
     doi = {10.5802/aif.2751},
     zbl = {1271.58014},
     mrnumber = {3060760}
}

@article{Cardoso2002UniformEO,
  title={Uniform Estimates of the Resolvent of the Laplace-Beltrami Operator on Infinite Volume Riemannian Manifolds. II},
  author={F. Cardoso and G. Vodev},
  journal={Annales Henri Poincar{\'e}},
  year={2002},
  volume={3},
  pages={673-691}
}

@misc{tao2026spectralgapsurfacesinfinite,
      title={Spectral gap for surfaces of infinite volume with negative curvature}, 
      author={Z. Tao},
      year={2026},
      eprint={2403.19550},
      archivePrefix={arXiv},
      primaryClass={math.SP}
}

@article{BGH,
author = {Burq, N. and Guillarmou, C. and Hassell, A.},
year = {2009},
month = {07},
pages = {},
title = {Strichartz Estimates Without Loss on Manifolds with Hyperbolic Trapped Geodesics},
volume = {20},
journal = {Geometric and Functional Analysis},
doi = {10.1007/s00039-010-0076-5}
}

@article{spectralbookdyatlov,
author = {S. Dyatlov and M. Zworski},
year = {2019},
title = {Mathematical theory of scattering resonances},
volume = {200},
journal = {American
Mathematical Soc.},
doi = {10.1007/s00039-010-0076-5}
}

@article{Guillarmou2010RestrictionAS,
  title={Restriction and spectral multiplier theorems on asymptotically conic manifolds},
  author={C. Guillarmou and A. Hassell and A. S. Sikora},
  journal={Analysis \& PDE},
  year={2010},
  volume={6},
  pages={893-950}
}

@article{Stein,
  title={Oscillatory integrals in Fourier analysis},
  author={E. M. Stein},
  journal={Ann. of Math. Stud., Beijing Lectures in Harmonic Analysis, Princeton Univ. Press, Princeton, NJ,},
  year={1986},
  volume={112},
  page={307–355}
}

@article{Thomas,
  title={A restriction theorem for the Fourier transforms},
  author={P. A. Tomas},
  journal={Bull. Amer. Math. Soc},
  year={1975}
}

@book{SoggeHangzhou,
    author = {Sogge, C.D.},
    title = {Hangzhou Lectures on Eigenfunctions of the Laplacian},
    publisher = {Annals of Mathe-
matics Studies., vol. 188. Princeton University Press, Princeton},
    year = {2014}
}

@article{HormanderIII,
    author = {L. H\"ormander},
    title = {The analysis of linear partial differential operators. III. },
    journal = {Classics in Mathematics.
Springer, Berlin, Pseudo-differential operators},
    year = {2007}
}

@article{xichen,
    author = {Chen, X.},
    title = {Stein–Tomas restriction theorem via spectral measure on metric measure spaces.},
    journal = {Math. Z. 289, 829–835, doi.org/10.1007/s00209-017-1976-y},
    year = {2018}
}

@article{bouclet,
    author = {J.-M. Bouclet.},
    title = {Strichartz estimates on asymptotically hyperbolic manifolds.},
    journal = {Anal. PDE, 4(1):1–84},
    year = {2011}
}

@article{sogge88,
    author = {C.D. Sogge},
    title = {Concerning the $L^p$ norm of spectral clusters for second-order elliptic operators on compact manifolds.},
    journal = {Journal of Functional Analysis, Vol. 77, Issue 1, 123-138},
    year = {1988}
}

@article{DVgluing,
  author = {Datchev, Kiril and Vasy, Andr{\'a}s},
  title = {Gluing semiclassical resolvent estimates via propagation of singularities},
  journaltitle = {International Mathematics Research Notices},
  year = {2012}, number = {23}, pages = {5409--5443},
  eprint = {1008.3964}, eprinttype = {arxiv}, eprintclass = {math.AP}
}

@article{VasyZworski2000,
  author = {Vasy, Andr{\'a}s and Zworski, Maciej},
  title = {Semiclassical estimates in asymptotically {Euclidean} scattering},
  journaltitle = {Communications in Mathematical Physics},
  year = {2000}, volume = {212}, pages = {205--217},
  doi = {10.1007/s002200000207}
}

@online{WangAH,
  author = {Wang, Yiran},
  title = {Resolvent and Radiation Fields on Non-trapping Asymptotically Hyperbolic Manifolds},
  year = {2014}, eprint = {1410.6936}, eprinttype = {arxiv}, eprintclass = {math.AP}
}
\end{document}